\documentclass[11pt,a4paper,reqno]{amsart}
\usepackage[top=1.15in, bottom=1.15in, left=1.17in, right=1.17in]{geometry}
\usepackage[colorlinks=true, pdfstartview=FitV, linkcolor=blue, citecolor=blue, urlcolor=blue, breaklinks=true]{hyperref}

\usepackage{hyphenat}
\usepackage{amssymb}
\usepackage{mathrsfs}
\usepackage{todonotes}
\usepackage{mathtools}
\usepackage{thmtools}
\usepackage{amscd}
\usepackage{todonotes}
\usepackage{enumitem}
\usepackage{comment}
\usepackage{graphicx}
\usepackage{dsfont}
\usepackage{nameref, cleveref}
\usepackage{fancyhdr, color,tikz-cd}
\usepackage[table]{xcolor}
\usepackage{subcaption}

\makeatletter
\newcommand*{\@old@slash}{}\let\@old@slash\slash
\def\slash{\relax\ifmmode\delimiter"502F30E\mathopen{}\else\@old@slash\fi}
\makeatother

\newcommand{\g}{\mathfrak{g}}

\newcommand{\uq}{U_q(\mathfrak{sl}_r)}

\newcommand{\wt}{\widetilde}
\newcommand{\weight}{\mathsf{wt}}

\newcommand{\C}{\mathbb{C}}

\renewcommand{\L}{\mathcal{L}}
\renewcommand{\o}{\circ}

\renewcommand{\b}{\mathfrak{b}}

\newcommand{\h}{\mathfrak{h}}

\newcommand{\W}{\mathcal{W}}

\renewcommand{\deg}{\text{deg}}

\DeclareMathOperator{\SL}{SL}

\numberwithin{equation}{section}
\newcommand{\swap}{\mathsf{swap}}
\newcommand{\flip}{\mathsf{flip}}
\newcommand{\trip}{\mathsf{trip}}
\newcommand{\sep}{\mathsf{sep}}
\newcommand{\osc}{\mathsf{osc}}
\newcommand{\type}{\mathsf{type}}
\newcommand{\inv}{\mathsf{Inv}}
\renewcommand{\hom}{\mathsf{Hom}}

\newcommand{\CRG}{\mathsf{CRG}}
\newcommand{\CG}{\mathsf{CG}}
\newcommand{\TCRG}{\mathsf{TCRG}}
\newcommand{\SSV}{\mathsf{SSV}}
\newcommand{\WSSV}{\mathsf{WSSV}}
\newcommand{\BL}{\mathsf{BL}}
\newcommand{\RT}{\mathsf{RT}}

\newtheorem{thm}{Theorem}[section]
\newtheorem*{thm*}{Theorem}
\newtheorem{corollary}[thm]{Corollary}
\newtheorem{lemma}[thm]{Lemma}
\newtheorem{proposition}[thm]{Proposition}

\newtheorem{example}[thm]{Example}

\newtheorem{algorithm}[thm]{Algorithm}

\theoremstyle{definition}
\newtheorem{definition}[thm]{Definition}

\theoremstyle{remark}
\newtheorem{remark}[thm]{Remark}

\crefname{lemma}{lemma}{lemmas}
\Crefname{Lemma}{Lemma}{Lemmas}

\title{Clasped web bases from hourglass plabic graphs}

\author[Enugandla]{Pranav Enugandla}
\address[Enugandla]{University of California, Berkeley, CA, United States of America}
\email{shreepranav\_varma\@berkeley.edu}

\author[Gaetz]{Christian Gaetz}
\thanks{The authors were supported by the National Science Foundation under award no. DMS-2452032 and by a travel grant from the Simons Foundation.}
\address[Gaetz]{University of California, Berkeley, CA, United States of America}
\email{gaetz\@berkeley.edu}

\date{\today}

\begin{document}

\begin{abstract}
G.--Pechenik--Pfannerer--Striker--Swanson \cite{GPPSS-sl4} applied \emph{hourglass plabic graphs} to construct \emph{web bases} for spaces of tensor invariants of fundamental representations of $U_q(\mathfrak{sl}_4)$, extending Kuperberg's celebrated basis for $U_q(\mathfrak{sl}_3)$ \cite{Kuperberg}. We give several combinatorial characterizations of basis webs in the kernel of the projection to invariants in a tensor product of arbitrary (type $1$) irreducibles. We apply this to show that the nonzero images of basis webs form a basis (a property shared with Lusztig's dual canonical basis) yielding distinguished \emph{clasped} web bases for each such tensor product.
\end{abstract}

\maketitle
\section{Introduction}

\subsection{Webs and web bases}

Consider a tensor product $V_q^{\underline{a}} = \bigotimes_{i=1}^n V_q(\omega_{a_i})$ of irreducible complex representations of $\uq$ indexed by fundamental weights $\omega_1,\ldots, \omega_{r-1}$. \emph{Webs}, introduced by Kuperberg \cite{Kuperberg} as a tool for the computation of quantum link invariants, give a diagrammatic calculus for $\hom_{U_q(\mathfrak{sl}_r)}(V_q^{\underline{a}'},V_q^{\underline{a}''})$ (and its quantum group deformation). By dualizing and moving tensor factors, we may equally well choose to study the space of tensor invariants:
\[
\inv(V_q^{\underline{a}}) \coloneqq \hom_{\uq}(V_q^{\underline{a}},\mathbb{C}(q)).
\]

Webs $W$ representing elements $[W]_q \in \inv(V_q^{\underline{a}})$ are planar bipartite graphs, whose edges are colored by fundamental weights, and which are embedded in a disk whose boundary vertices $b_1,\ldots,b_n$ are each incident to a single edge, colored $\omega_{a_i}$. For this reason, we call $\underline{a}$ the \emph{boundary conditions} of the web. We follow the conventions (of e.g. \cite{Sikora} and \cite{Fraser-Lam-Le}) that the sum of the indices of the fundamental weights coloring the edges incident to each internal vertex is $r$.

There are in general many relations (see \cite{Cautis-Kamnitzer-Morrison}) between the invariants $[W]_q$ of webs with boundary conditions $\underline{a}$. A \emph{web basis} is a subset of the web invariants forming a basis. Kuperberg gave a $U_q(\mathfrak{sl}_3)$ web basis consisting of the \emph{non-elliptic} webs. Beyond early applications for quantum link invariants \cite{Khovanov}, the non-elliptic basis has also found application to skein modules \cite{Le.Sikora}, dimer models \cite{Douglas-Kenyon-Shi}, and dynamical algebraic combinatorics \cite{Petersen-Pylyavskyy-Rhoades}, among other areas. 

Since Kuperberg's work, much effort has gone into constructing higher-rank web bases, with additional special properties. In \cite{GPPSS-sl4}, the first \emph{rotation-invariant} $U_q(\mathfrak{sl}_4)$ web basis was constructed using \emph{hourglass plabic graphs}. This is the basis of \emph{(top) fully reduced} web invariants $\W_{\underline a}$. Fully reduced hourglass plabic graphs also recover Kuperberg's basis for $U_q(\mathfrak{sl}_3)$.

An obvious limitation of the fully reduced web bases is that they are heretofore only known for tensor products of irreducibles indexed by fundamental weights. To extend this basis to arbitrary tensor products of irreducibles, we need to understanding the \emph{clasping} of these webs. Kuperberg gave clasping rules for $U_q(\mathfrak{sl}_3)$, and these clasped webs have appeared, for example, in well-known conjectures of Fomin--Pylyavskyy \cite{Fomin-Pylyavskyy-advances} relating basis webs to cluster algebras.

\subsection{Clasped webs}
\emph{Clasped webs} provide for an extension of the diagrammatic calculus to spaces of morphisms between tensor products of general finite-dimensional (type-$1$) irreducible representations $V_q(\lambda)$.

Fix a partition of $[n]= I_1\sqcup \cdots \sqcup I_m$, with each $I_i$ an interval. This defines the \emph{clasp sequence} $\underline C = (\underline c_1,\dots, \underline c_m)$ where $\underline c_i = (a_j)_{j\in I_i} \in [r-1]^{|I_i|}$ is the $i$-th \emph{clasp}. Oftentimes we will refer to the corresponding set of boundary vertices $\{b_j\ |\ j\in I_i\}$ also as the $i$-th clasp, with $\underline c_i$ recording the tuple of boundary conditions of the vertices in the clasp. The weight of the clasp $\underline c_i$ is defined to be $\weight(\underline c_i) = \sum_{j\in I_i}\omega_{a_j}\in \Lambda^+,$ where $\Lambda^+$ is the set of dominant integral weights for $G=\SL_r(\mathbb{C})$.

Fix a clasp sequence $\underline C = (\underline c_1,\dots, \underline c_m)$ for boundary conditions $\underline a$, and let $\lambda_i = \weight(\underline c_i)$. Recall that for each $i$ there is an inclusion (unique up to scaling) of $\uq$-representations $V_q(\lambda_i)\xhookrightarrow{} V_q^{\underline c_i}.$ This induces an inclusion of $\uq$-representations
\[
\bigotimes_{i=1}^m V_q(\lambda_i) \xhookrightarrow{} \bigotimes_{i=1}^m V_q^{\underline c_i} = V_q^{\underline a}.
\]
Therefore, we obtain a surjective map of invariant spaces
\[\pi_{\underline C}:\inv(V_q^{\underline a}) \to \inv\left(\bigotimes_{i=1}^m V_q(\lambda_i)\right).\]
The image $\pi_{\underline{C}}([W]_q)$ of a web invariant $[W]_q \in \inv(V_q^{\underline{a}})$ is a \emph{clasped web invariant}.

\subsection{Webs as hourglass plabic graphs}

\emph{Hourglass plabic graphs} are a combinatorial manifestation of webs, introduced by G.--Pechenik--Pfannerer--Striker--Swanson \cite{two-column, GPPSS-sl4}; see \Cref{sec:hpg-background}. In these, edges colored by $\omega_{a}$ are drawn as twisted ``hourglass" edges of $a$ strands. The fundamental combinatorial data associated to an hourglass plabic graph are the \emph{trips}: certain walks along its edges. These trips can be used to characterize the graphs appearing in the fully reduced web basis. As we see in our main result\footnote{The equivalence of (1) and (2) for $r=2,3$ was shown by Kuperberg \cite{Kuperberg}. The equivalence of these with (3) for $r=2,3$ will appear in independent forthcoming work of Catania, Kim, and Pfannerer \cite{catania-kim-pfannerer}.} below, they remarkably also exactly characterize the kernel of the clasping map $\pi_{\underline C}$.

\begin{thm}
    \label{thm:main}
    Let $r \in \{2,3,4\}$, let $\W_{\underline a}$ be the fully reduced web basis for $\inv(V_q^{\underline a})$, and let $\underline{C}$ be a clasp sequence for $\underline{a}$ with weights $\lambda_1,\ldots,\lambda_m$. Then the following are equivalent for $[W]_q \in \W_{\underline a}$:
    \begin{enumerate}
        \item $[W]_q\notin \ker\pi_{\underline C}$,
        \item $W$ is non-convex (see \Cref{def:nonconvex}),
        \item $W$ has no trips that start and end in the same clasp.
    \end{enumerate}
    Moreover, the clasped web invariants for these webs $W$ form a basis for $\inv\left(\bigotimes_{i=1}^m V_q(\lambda_i)\right)$. 
\end{thm}

\begin{remark}
    The fact that the web invariants not killed by the projection $\pi_{\underline C}$ form a basis in $\inv\left(\bigotimes_{i=1}^m V_q(\lambda_i)\right)$ is notable. This property is one of several special properties (also including rotation-invariance) that the fully reduced web bases share with Lusztig's dual canonical basis \cite{Lusztig:canonical}. It was initially hoped that Kuperberg's basis agreed with the dual canonical basis until this was disproven by Khovanov--Kuperberg \cite{Khovanov-Kuperberg}, however the two bases do seem to share many distinguished properties. 
\end{remark}

In the case when $\underline C$ is \emph{sorted}, i.e., each $\underline c_i$ is a weakly increasing tuple, the following theorem gives us more criteria to easily check for non-convexity.

\begin{thm}
    \label{thm:intro-sorted}
    In the setting of \Cref{thm:main}, assume further that $\underline C$ is sorted. Then the following additional conditions are equivalent to \Cref{thm:main}(1)-(3):
    \begin{enumerate}
        \item[(4)] No clasp contains a ``bad" local configuration (see \Cref{fig:bad_configs,fig:badconf_3,fig:u_turn}).
        \item[(5)] The lattice word $\L(W) = \partial\mathrm{sep}(W)$ has no $\underline C$-descents (see \Cref{def:latticeword}).
    \end{enumerate}
\end{thm}

The map $\partial\mathrm{sep}$ appearing in \Cref{thm:intro-sorted}(5) is the bijection from \cite{GPPSS-sl4} between elements of $\W_{\underline a}$ and (lattice words of) the associated \emph{fluctuating tableaux} \cite{fluctuating-paper}, certain generalizations of standard Young tableaux (which correspond to the case $\underline{a}=(1,1,\ldots,1)$).

A key tool in the proof of \Cref{thm:main} is the \emph{swap} map of \Cref{sec:sorting}, which is of interest even outside the context of clasping. It gives bijections between the fully reduced web bases of \cite{GPPSS-sl4} for any ordering of the same tensor factors.

\begin{thm}
\label{thm:intro-swap}
    The map $\swap$ gives bijections between the fully reduced web bases $\W_{\underline a}$ for all permutations of the boundary conditions $\underline{a}$.
\end{thm}

For ease of exposition, in most of the paper we deal only with the classical groups $\SL_r$. All of our results apply also to invariants of the quantum group $U_q(\mathfrak{sl}_r)$, but this extension will be straightforward. 

\subsection{Outline}

In \Cref{sec:prelim} we recall background on web invariants and on the combinatorial constructions introduced in \cite{GPPSS-sl4}. In \Cref{sec:clasped-webs} we set up some machinery for clasping webs. In \Cref{sec:descents} we establish a correspondence between certain descents in a lattice word and bad local configurations in the corresponding basis web. In \Cref{Sec:Fund} we show that lattice words avoiding these descents (and therefore webs avoiding the bad configurations) have the right number to form a basis for the clasped invariant space. This is applied in \Cref{sec:main-sorted-proof} to prove \Cref{thm:intro-sorted}. \Cref{sec:sorting} introduces the swap map which is used to prove \Cref{thm:intro-swap} and thereby \Cref{thm:main}, after some verifications for $r=2$ and $3$, which take place in \Cref{sec:r-2-3}.

\section{Preliminaries}
\label{sec:prelim}

\subsection{Tensor invariant spaces}

Let $G$ be the group $\SL_r(\C)$ and $\g$ its Lie algebra, with Cartan subalgebra $\h\subset \g$ given by traceless diagonal matrices, and Borel subalgebra $\b\subset \g$ given by traceless upper triangular matrices. Let $\Lambda^+\subset \h^*$ denote the set of dominant integral weights for $G$, and for each $\lambda\in \Lambda^+$, let $V(\lambda)$ denote the irreducible finite-dimensional representation for $G$ with highest weight $\lambda$. In particular, for the fundamental weights $\omega_1, \dots, \omega_{r-1}$, we have
\[V(\omega_k)\simeq \bigwedge^k V, \hspace{5mm} 1\le k\le r-1\]
where $V = \C^r$ is the defining representation of $G$. 

\begin{definition}
    Let $[r] = \{1, 2, \dots, r\}$ and $\overline{[r]} = \{\overline 1, \overline2, \dots, \overline r\}$. Given a \emph{type} $\underline a = (a_1, \dots, a_n)$ with $a_i\in [r-1]$, let
    \[V^{\otimes \underline a} \coloneqq 
    \bigotimes_{i=1}^n V(\omega_{a_i})
    \simeq \bigotimes_{i=1}^n \left(\bigwedge^{a_i}V\right).\]
\end{definition}

For $r = 4$, G.--Pechenik--Pfannerer--Striker--Swanson \cite{GPPSS-sl4} gave a rotation invariant diagrammatic basis for the invariant space $\inv_G(V^{\otimes \underline a}) \coloneqq \hom_G(V^{\otimes \underline a}, \C)$, where $\C$ denotes the trivial $G$-representation, in terms of \emph{hourglass plabic graphs}.

\subsection{Hourglass plabic graphs}
\label{sec:hpg-background}

For the rest of the section, we specialize to $r = 4$. Our conventions and terminology differ only slightly from those of \cite{GPPSS-sl4}.

\begin{definition}
    An \emph{hourglass plabic graph} is a planar bipartite graph $W$ embedded in the disk with a fixed black-white vertex coloring and with boundary vertices labeled clockwise as $b_1, \dots, b_n$, such that:
    
    \begin{itemize}
        \item each edge $e$ has a \emph{multiplicity} $m(e)\in \{1,2\}$,
        \item each internal vertex has degree $4$ (counted with multiplicity), and
        \item each boundary vertex is adjacent to exactly one edge.
    \end{itemize}   
\end{definition}

Edges with $m(e)=1$ are called \emph{simple edges} and edges with $m(e)=2$ are called, and drawn as, \emph{hourglass edges} (see \Cref{fig:web}). The \emph{simple degree} of a vertex is its number of incident edges (\emph{ignoring} multiplicity). We consider $W$ up to planar isotopy fixing the boundary circle of the disk. The boundary face between $b_n$ and $b_1$ is called the \emph{base face} and denoted $F_0$.

\begin{definition}
    The \emph{type} of an hourglass plabic graph $W$ is the tuple $\type(W) = \underline a = (a_1,\dots, a_n)$ where
    \[a_i = \begin{cases}
        1 & \text{if }\deg(b_i) = 1\text{ and }b_i \text{ is black},\\
        2 & \text{if }\deg(b_i) = 2,\\
        3 & \text{if }\deg(b_i) = 1\text{ and }b_i \text{ is white}.
    \end{cases}\]
    
    $W$ is of \emph{oscillating type} if $a_i\in \{1, 3\}$ for each $1\le i\le n$, i.e., if $\deg(b_i) = 1$ for each $1\le i\le n$. A type written $\underline o$ is assumed to be oscillating.
\end{definition}

\begin{figure}[h]
    \centering
    \includegraphics[width=0.35\linewidth]{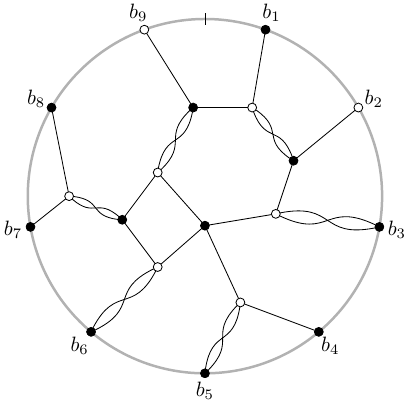} \hspace{.5in}\includegraphics[width=0.35\linewidth]{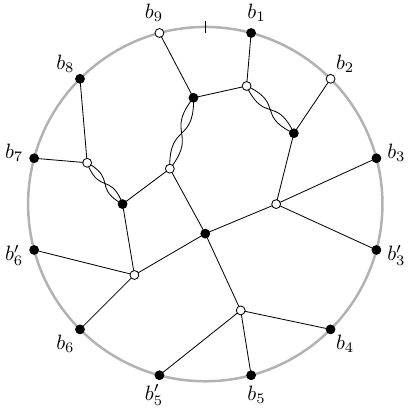}
    \caption{\textsc{(left)} An hourglass plabic graph of type $(1,3,2,1,2,2,1,1,3)$, and  \textsc{(right)} its oscillization.}
    \label{fig:web}
\end{figure}

Given an hourglass plabic graph of arbitrary type, there is a natural way to modify it to one of oscillating type. This oscillization procedure often allows us to reduce to the oscillating case.

\begin{definition}
    Let $W$ be an hourglass plabic graph, with boundary vertex $b_i$ connected to $v$ via an hourglass edge. Then the \emph{partial oscillization} of $W$ is the hourglass plabic graph $\osc_i(W)$ obtained by splitting $b_i$ into two boundary vertices $b_i$ and $b_i'$ which are both connected to $v$ via simple edges. The \emph{oscillization} of $W$, denoted $\osc(W)$, is obtained by performing all possible partial oscillizations. 
\end{definition}

In \cite{GPPSS-sl4}, the authors explain how to associate $\SL_4$ invariants to hourglass plabic graphs using \emph{proper labelings}: assignments of an $m(e)$-subset of $[4]$ to each edge such that the union around each internal vertex is $[4]$. Let $W$ be an hourglass plabic graph of oscillating type $\underline o$. This defines an invariant $[W]\in \inv_{\SL_4}(V^{\otimes \underline o})$, which is given in coordinates, up to sign, by
\[[W] = \sum_\phi (-1)^{\mathrm{sgn}(\phi)}x_{\partial(\phi)},\]
where the sum runs over all proper labelings $\phi$ of the edges of $W$. The monomial $x_{\partial(\phi)}$ records the colors of the boundary edges. If $W$ is of general type $\underline a$, then we obtain $[W]\in \inv_{\SL_4}(V^{\otimes \underline a})$ by just restricting $[\osc(W)]$ to $V^{\otimes \underline a}$. 

Hourglass plabic graphs are just combinatorial manifestations of webs, with the hourglass edges corresponding to $\omega_2$. We therefore use these words interchangeably. Nevertheless, the applicability of the combinatorics of \emph{moves} and \emph{trips} to algebraic questions is central to our work.

\subsection{Moves}
Hourglass plabic graphs admit \emph{contraction} moves, \emph{square} moves, and \emph{benzene} moves, as shown in \Cref{fig:contraction,fig:square,fig:benzene}.

\begin{figure}[h]
    \centering
    \includegraphics[width=0.5\linewidth]{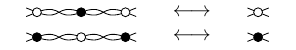}
    \caption{Contraction moves.}
    \label{fig:contraction}
\end{figure}

\begin{figure}[h]
    \centering
    \includegraphics[width=\linewidth]{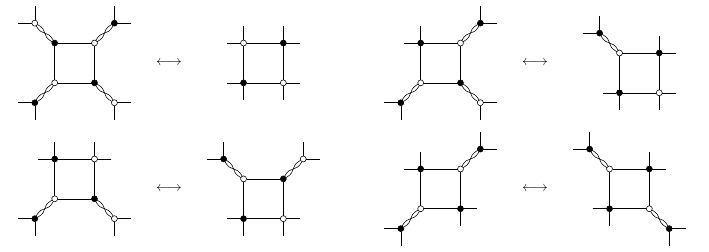}
    \caption{Square moves.}
    \label{fig:square}
\end{figure}

\begin{figure}[h]
    \centering
    \includegraphics[width=0.5\linewidth]{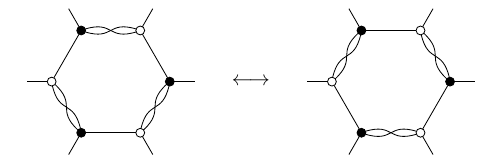}
    \caption{The benzene move.}
    \label{fig:benzene}
\end{figure}

\begin{definition}
    Two hourglass plabic graphs $W, W'$ are \emph{move equivalent}, written $W\sim W'$, if some sequence of contraction, square, and benzene moves transforms $W$ into $W'$. We say $W$ is \emph{contracted} if contraction moves have been applied to ensure there are no pairs of adjacent hourglass edges, except possibly at the boundary. We denote the set of contracted hourglass plabic graphs of type $\underline a$ by $\CG(\underline a)$. 
\end{definition}

\begin{remark}
    Contraction moves and square moves preserve the invariant corresponding to the hourglass plabic graph, but benzene moves do not (see \cite[\S~7.1]{GPPSS-sl4}).
\end{remark}

In analogy to the $\SL_3$ case and to Postnikov's plabic graphs \cite{Postnikov-ICM}, the \emph{full reducedness} of hourglass plabic graphs is characterized by forbidding certain local configurations.

\begin{definition}
    An hourglass plabic graph $W$ is called \emph{fully reduced} if it has no isolated components, and if no $W'\sim W$ contains a 4-cycle with an hourglass edge. We denote the set of contracted fully reduced hourglass plabic graphs of type $\underline a$ by $\CRG(\underline a)$.
\end{definition}

Full reducedness is precisely the restriction that cuts down the spanning set of all webs to a basis for $\inv_G(V^{\otimes \underline a})$, after choosing suitable representatives form each move-equivalence class. We say $W \in \CRG(\underline a)$ is \emph{top} if in each hourglass edge of a benzene face the white vertex precedes the black vertex in clockwise order. Let $\TCRG(\underline a)\subset \CRG(\underline a)$ be the set of top fully reduced hourglass plabic graphs of type $\underline a$.

\begin{thm}[Thm.~A of \cite{GPPSS-sl4}]\label{thm:web_basis}
     For any type $\underline{a} \in [r-1]^n$, the set of invariants $\W_{\underline a} \coloneqq \{[W]: W\in \TCRG(\underline a)\}$ forms a basis for $\inv_G(V^{\otimes \underline a})$.
\end{thm}

\subsection{Trips in hourglass plabic graphs} In light of \Cref{thm:web_basis}, we would like a criterion for the full reducedness of a given web $W$ that does not require exploring the move-equivalence class. Such a criterion is obtained by analyzing the \emph{$\trip_\bullet$-strands} of $W$.

\begin{definition}
    Let $W$ be an hourglass plabic graph and $b_j$ a boundary vertex. For $i\in[3]$, the \emph{$\trip_i$-strand through $b_j$} is defined as the walk on $W$ starting from $b_j$ and following the \textit{rules of the road} until ending at some other boundary vertex (see \Cref{fig:trip-example}):
    \begin{itemize}
        \item At a black vertex, take the $i$-th rightmost turn;
        \item At a white vertex, take the $i$-th leftmost turn.
    \end{itemize}
    
    When the boundary vertex $b_j$ is incident to a simple edge, there is no ambiguity as to what the first edge of the walk should be. However if $b_j$ is incident to an hourglass, there are two choices for the first edge of the walk. When $i=1$ or $i=3$, one of these choices leads to a walk that ``bounces back" to $b_j$, so we pick the starting edge that does not do so. Thus there is one $\trip_1$- and one $\trip_3$-strand through such a vertex, but two $\trip_2$-strands. We write $\trip_i(W)$ for the resulting map from boundary vertices to (pairs of) boundary vertices and $\trip_{\bullet}(W)$ for the tuple of these maps.
\end{definition}

\begin{figure*}[h!]
    \centering
    \begin{subfigure}[h]{0.5\textwidth}
        \centering
        \includegraphics[width=0.7\linewidth]{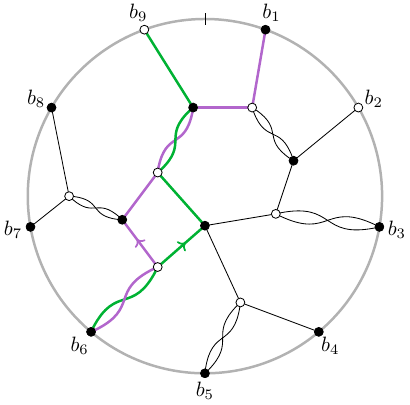}
        \caption{$\trip_1$- in green, $\trip_3$- in purple}
    \end{subfigure}%
    ~ 
    \begin{subfigure}[h]{0.5\textwidth}
        \centering
        \includegraphics[width=0.7\linewidth]{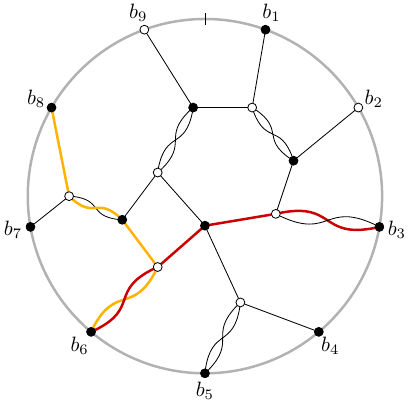}
        \caption{$\trip_2$-strands in orange and red}
    \end{subfigure}
    \caption{Trips through $b_6$.}
\label{fig:trip-example}
\end{figure*}

\begin{remark}
    $\trip_3$-strands are just $\trip_1$-strands traversed in reverse. Thus, we can often make statements about $\trip_3$-strands in terms of $\trip_1$-strands, especially when the direction of traversal is irrelevant. Moreover, $\trip_2$-strands are the same when traversed in either direction, so we do not specify this direction in figures.  
\end{remark}

Equipped with the notion of trips, we can state the criterion for full reducedness in terms of monotonicity of trips. We will not need this criterion in this paper, as we shall use the one in terms of symmetrized six-vertex configurations instead (\Cref{prop:well_oriented}), but we include it here for completeness.

\begin{definition}
    An hourglass plabic graph $W$ is \emph{monotonic} if $\trip_2$-strands have no self intersections or double crossings, and if for every $\trip_1$-strand $\ell_1$ and $\trip_2$-strand $\ell_2$, the vertices lying on both $\ell_1$and $\ell_2$ are consecutive along both $\ell_1$ and $\ell_2$.
\end{definition}

\begin{thm}[Thm.~3.13 \& Cor.~3.33 of \cite{GPPSS-sl4}]
    An hourglass plabic graph $W$ is fully reduced if and only if it is monotonic. Moreover if $W, W'\in \CRG(\underline a)$, then $W\sim W'$ if and only if $\trip_\bullet(W)=\trip_\bullet(W')$.
\end{thm}

\subsection{Separation labeling}

Given a contracted fully reduced hourglass plabic graph $W$ of oscillating type, there is a distinguished proper labeling of its edges called the \emph{separation labeling} $\sep_W$ defined as follows:

\begin{itemize}[leftmargin=*]
    \item For each simple edge $e$ of $W$, let $F(e)$ be the face incident to $e$ which is to the right of $e$ when traversed from the black vertex to the white vertex. Then
    \[\sep_W(e) \coloneqq 1+\big|S(e)\big|,\]
    where $S(e)\subset [3]$ is the set of all $k$ such that the $\trip_k$-strand through $e$ separates $F_0$ from $F(e)$.
    \item For each hourglass edge $e$ of $W$, let $v$ be either of the endpoints of $e$, and let $e', e''$ be other two (simple) edges of $W$ incident to $v$. Then 
    \[\sep_W(e) \coloneqq [4]\setminus \{\sep_W(e'), \sep_W(e'')\}.\]
\end{itemize}

For a contracted fully reduced hourglass plabic graph $W$ of general type, we define $\sep_W$ by first computing $\sep_{\osc(W)}$, labeling each internal edge and boundary simple edge of $W$ by the corresponding label in $\sep_{\osc(W)}$, and each boundary hourglass edge by the set of labels appearing on the corresponding pair of boundary simple edges of $\osc(W)$. Note that this definition makes sense with our slightly more general notion of contracted fully reduced hourglass plabic graphs, where we allow length $2$ chains of hourglass edges at the boundary, because the oscillization is still a contracted fully reduced hourglass plabic graph of oscillating type in the sense of \cite{GPPSS-sl4}.

Let $b_1,\dots, b_n$ be the boundary vertices of $W$ with incident edges $e_1,\dots, e_n$ respectively. Then the separation labeling of $W$ gives a word $\L(W) \coloneqq \partial (\sep_W) = w_1\cdots w_n$ in the letters $[4]\sqcup {[4]\choose 2}\sqcup \overline{[4]}$, with 
\[w_i = \begin{cases}
    \overline{\sep_W(e_i)} & \text{if } a_i = 3,\\
    \sep_W(e_i) & \text{if } a_i = 1\text{ or } 2.\\ 
\end{cases}\]

\begin{definition}\label{def:sortedword}
    Let $L = w_1\cdots w_n$ be a word in the letters $[4]\cup \overline{[4]}\cup {[4]\choose 2}$. We say

    \[ |w_i| = 
    \begin{cases}
        1 & \text{if } w_i\in [4]\\
        2 & \text{if } w_i\in {[4]\choose 2}\\
        3 & \text{if } w_i\in \overline{[4]}
    \end{cases}
    \]
    The \emph{type} of $L$ is $\type(L) = (|w_1|,\dots, |w_n|),$ and we say that $L$ is \emph{sorted} if $|w_1|\le \cdots \le |w_n|$.
\end{definition}

\begin{remark}
    It follows from the definition that $\type(\L(W)) = \type(W)$.
\end{remark}

Recall that a \emph{lattice word} $L=w_1\cdots w_n$ is a word in which for each prefix $w_1\cdots w_i$, the number of occurrences of $a$ is greater than or equal to the number of occurrences of $b$ for all $a<b\in[4]$ (barred letters count as $-1$ appearance of the corresponding unbarred letter). For example, $1\{2, 3\}\overline 4$ is a lattice word, but $1\overline 2\{3, 4\}$ is not. A lattice word is \emph{balanced} if the number of occurrences of each element of $[4]$ are equal. For example, $1\{2, 3\}\{1, 2\} \overline 2 4\overline 1$ is balanced, whereas $1\{1,2\}3\{2, 4\} \overline 4$ is not. We have:

\begin{thm}[Thm.~4.24 of \cite{GPPSS-sl4}]
    If $W\in \CRG(\underline a)$, then $\L(W)$ is a balanced lattice word.
\end{thm}

\subsection{Symmetrized six-vertex configurations}

Hourglass plabic graphs of oscillating type are in bijection with \emph{symmetrized six-vertex configurations}. The latter are better suited for some of our arguments, so we recall the definitions and the bijection here.

\begin{definition}[See \cite{GPPSS-sl4, Hagemeyer}]
    A \emph{symmetrized six-vertex configuration} $D$ is a planar directed graph embedded in a disk with boundary vertices $b_1,\dots, b_n$ (labeled clockwise) of degree $1$, such that each internal vertex is some rotation of a vertex from \Cref{fig:ssv}.

    \begin{figure}[h]
        \centering
        \includegraphics[height=0.7in]{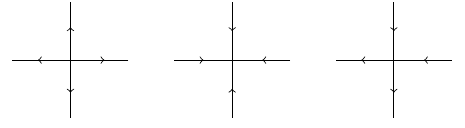}
        \label{fig:ssv}
    \end{figure}

    These vertices are called \emph{sources, sinks}, and \emph{transmitting vertices} respectively. The \emph{type} of $D$ is $\type(D) = \underline o = (o_1,\dots, o_n)$ where
    \[o_i = \begin{cases}
        1 & \text{if the edge incident to }b_i\text{ is oriented inwards,}\\
        3 & \text{if the edge incident to }b_i\text{ is oriented outwards.}
    \end{cases}\]
    We denote the set of all symmetrized six-vertex configurations of type $\underline o$ by $\SSV(\underline o)$.
\end{definition}

Symmetrized six-vertex configurations admit \emph{Yang--Baxter} and \emph{ASM} moves, which are analogs of benzene moves and square moves for hourglass plabic graphs.

\begin{figure*}[h!]
    \centering
    \begin{subfigure}[h]{0.5\textwidth}
        \centering
        \includegraphics[height=0.8in]{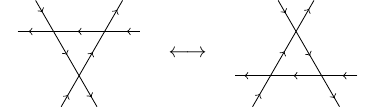}
        \caption{Yang--Baxter move}
    \end{subfigure}%
    ~ 
    \begin{subfigure}[h]{0.5\textwidth}
        \centering
        \includegraphics[height=0.8in]{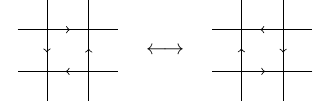}
        \caption{ASM move}
    \end{subfigure}
    \caption{Moves on symmetrized six-vertex configurations. For the ASM move, omitted orientations can be picked in any way yielding valid vertices.}
\end{figure*}

\begin{definition}
    Two symmetrized six-vertex configurations $D, D'$ are \emph{move equivalent}, denoted $D\sim D'$, if some sequence of ASM and Yang--Baxter moves transfroms $D$ into $D'$.
\end{definition}

The analog of full reducedness for symmetrized six-vertex configurations is \emph{well-orientedness}.

\begin{definition}
    A symmetrized six-vertex configuration $D$ is \emph{well-oriented} if it is simple (no loops or multiple edges), has no isolated components, and every 3-cycle in any $D'\sim D$ is cyclically oriented. We denote the set of well-oriented symmetrized six-vertex configurations of type $\underline o$ by $\WSSV(\underline o)$.
\end{definition}

As with full reducedness for hourglass plabic graphs, \cite{GPPSS-sl4} gives a criterion for checking well-orientedness of a symmetrized six-vertex configuration without exploring the entire move-equivalence class.

\begin{definition}
    A \emph{$\trip_2$-strand} of $D\in \SSV(\underline o)$ is a walk in $D$ obtained by starting at a boundary vertex of $D$ and walking along edges by going straight across to the opposite edge at each internal vertex, until reaching the boundary again. 
\end{definition}

\begin{proposition}[Lem.~3.20 \& Prop.~3.21 of \cite{GPPSS-sl4}]\label{prop:well_oriented}
    A symmetrized six-vertex configuration $D$ with no isolated components is well-oriented if and only if it satisfies the following two properties:
    \begin{itemize}
        \item[(P1)] The $\trip_2$-strands of $D$ do not have self intersections or double crossings.
        \item[(P2)] Big triangles are oriented, i.e., if $\ell_1,\ell_2,\ell_3$ are pairwise intersecting $\trip_2$ strands then the boundary of the triangle formed by these strands is cyclically oriented.
    \end{itemize}
    Furthermore, any well-oriented configuration $D$ satisfies the additional property:
    \begin{itemize}
        \item[(P3)] In any collection of 4 $\trip_2$-strands of $D$, there is some pair of strands that do not intersect.
    \end{itemize}
\end{proposition}

Given an hourglass plabic graph $W$ of oscillating type $\underline o$, we obtain a symmetrized six-vertex configuration $\varphi(W)\in \SSV(\underline o)$ by orienting all edges from black vertices to white vertices and contracting every hourglass edge to a point. Additionally, $\varphi$ intertwines benzene moves with Yang--Baxter moves, square moves with ASM moves, and full reducedness with well-orientedness.

\begin{thm}[\cite{GPPSS-sl4}, Thm 3.25]\label{thm:six_vertex}
    The map $\varphi: \CG(\underline o)\to \SSV(\underline o)$ is a bijection that intertwines benzene moves with Yang--Baxter moves and square moves with ASM moves. Moreover, $W\in \CRG(\underline o)$ if and only if $\varphi(W)\in \WSSV(\underline o)$.
\end{thm}

\begin{figure}[h]
    \centering
    \includegraphics[width=0.35\linewidth]{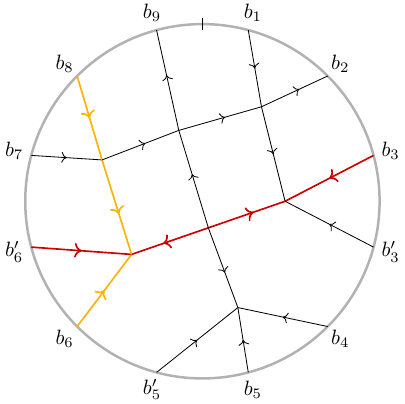}
    \caption{The symmetrized six-vertex configuration corresponding to the hourglass plabic graph from \Cref{fig:web}, with some $\trip_2$-strands highlighted.}
    \label{fig:web_ssv}
\end{figure}

\section{Clasped webs}
\label{sec:clasped-webs}

\subsection{Clasp sequences}
In this subsection, we set up the main definitions of clasp sequences and clasped webs. Fix boundary conditions $\underline a = (a_1,\dots, a_n)$.

\begin{definition}
    A \emph{clasp sequence} $\underline C = (\underline c_1, \dots, \underline c_m)$ on $\underline a$ is a partition of $[n]$ into $m$ disjoint intervals 
    $[n] = I_1\sqcup \cdots \sqcup I_m,$ with \emph{$i$-th clasp} $\underline c_i = (a_j)_{j\in I_i}\in [3]^{|I_i|}$. We will also refer to the set of boundary vertices $\{b_j\ |\ j\in I_i\}$ as the $i$-th clasp, with $\underline c_i$ recording the tuple of boundary conditions of the vertices in this clasp. The \emph{weight} of the clasp $\underline c_i$ is defined to be 
    \[\weight(\underline c_i) = \sum_{j\in I_i}\omega_{a_j}\in \Lambda^+.\]
\end{definition}

We picture a web $W\in \CRG(\underline a)$ along with a clasp sequence $\underline C$ by drawing external ``clasps" around corresponding boundary vertices, and call this a \emph{clasped web}.

\begin{figure}[h]
    \centering
    \includegraphics[width=0.35\linewidth]{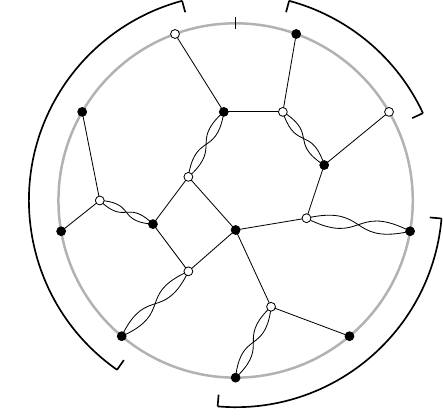}
    \caption{A clasped web with clasp sequence $\underline C = (\underline c_1, \underline c_2, \underline c_3)$, where $\underline c_1 = (1, 3), \underline c_2 = (2, 1, 2), \underline c_3 = (2, 1, 1, 3)$.}
    \label{fig:clasped_web}
\end{figure}

\begin{remark}\label{rem:flip}
    If $b_i$ is a boundary vertex of $W$ of type $2$, we may perform the transformation shown in \Cref{fig:flip} to \emph{flip} the color of $b_i$. This corresponds to the isomorphism of $G$-representations $\wedge^2 (V^\ast) \simeq \wedge^2(V)$. We call this transformation $\flip_i$. It is clear that $\flip_i$ is an involution that preserves non-convexity. Henceforth, we assume that all boundary vertices of type $2$ are colored black, unless otherwise mentioned.
\end{remark}
\begin{figure}[h]
        \centering
        \includegraphics[height=3cm]{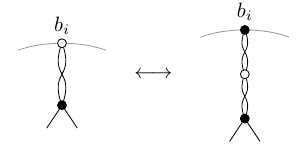}
        \caption{The flip transformation.}
        \label{fig:flip}
\end{figure}

\begin{definition}
    A clasp sequence $\underline C = (\underline c_1, \dots, \underline c_m)$ on $\underline a$ is \emph{sorted} if each $\underline c_i$ is a weakly increasing tuple.
\end{definition}

In other words, each clasp of a sorted clasp sequence first contains all type $1$ vertices, then all type $2$ vertices, and finally all type $3$ vertices, when read clockwise. If we assume that all type $2$ vertices are black (see \Cref{rem:flip}), then every clasp in a sorted clasp sequence looks like \Cref{fig:sorted} when read clockwise.

\begin{figure}[h]
    \centering
    \includegraphics[height = 1.5cm]{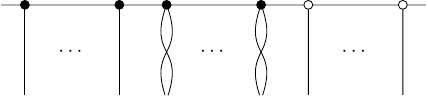}
    \caption{A sorted clasp.}
    \label{fig:sorted}
\end{figure}

\begin{remark}
    It is clear from the definition that if $\underline C = (\underline c_1, \dots, \underline c_m)$ is a clasp sequence on $\underline a$, then $\L(W) = L_1\cdots L_m$ with $\type(L_i) = \underline c_i$. If each $\underline c_i$ is a \emph{sorted} clasp, then each $L_i$ is a \emph{sorted} word (see \Cref{def:sortedword}).
\end{remark}

A clasp sequence $\underline C = (\underline c_1,\dots, \underline c_m)$ on $\underline a$ induces a (unique up to scaling) inclusion of $G$-representations:
\[\bigotimes_{i=1}^m V(\lambda_i)\xhookrightarrow{}V^{\otimes \underline a},\]
where $\lambda_i = \weight(\underline c_i)$. This inclusion induces a surjection on invariant spaces:
\[\pi_{\underline C}: \inv_G(V^{\otimes \underline a})\to \inv_G\left(\bigotimes_{i=1}^m V(\lambda_i)\right).\]

If we think of clasped webs as giving elements of the target invariant space, then we are essentially trying to describe all relations between the clasped webs. Since unclasped webs form a basis for $\inv_G(V^{\otimes\underline a})$, this is equivalent to understanding $\ker\pi_{\underline C}$.

\subsection{Cut paths and non-convexity}

We extend \cite{Kuperberg} in defining \emph{non-convexity} for webs. Let $\underline C = (\underline c_1, \dots, \underline c_m)$ be a clasp sequence on a fully reduced hourglass plabic graph $W$. A \emph{$\underline c_i$-cut path} $\gamma$ is an arc in the disk which separates the $i$-th clasp from the rest and is transverse to all edges of $W$ that it intersects. We define the weight of $\gamma$ as follows:

\begin{enumerate}[leftmargin=*]
    \item $r = 2$: $\weight(\gamma) = n\omega$ where $n$ is the number of edges of $W$ that $\gamma$ intersects.
    \item $r = 3$: $\weight(\gamma) = n_1\omega_1 + n_2\omega_2$ where $n_1$ (resp. $n_2$) is the number of edges of $e$ of $W$ that $\gamma$ intersects such that the black vertex (resp. white vertex) of $e$ is on the same side of $\gamma$ as the clasp.
    \item $r=4$: $\weight(\gamma) = n_1\omega_1 + n_2\omega_2 + n_3\omega_3$ where $n_1$ (resp. $n_3$) is the number of simple edges of $e$ of $W$ that $\gamma$ intersects such that the black vertex (resp. white vertex) of $e$ is on the same side of $\gamma$ as the clasp, and $n_2$ is the number of hourglass edges of $W$ that $\gamma$ intersects.
\end{enumerate}

\begin{definition}\label{def:nonconvex}
    A fully reduced hourglass plabic graph $W$ with a clasp sequence $\underline C = (\underline c_1,\dots, \underline c_m)$ is \emph{non-convex} (with respect to $\underline C$) if for each $i\in [m]$ and every $\underline c_i$-cut path $\gamma$, $\weight(\gamma)\ge \weight(\underline c_i)$ in the dominance ordering on $\Lambda^+$. We say that $W$ is \emph{partially convex} if it is not non-convex.
\end{definition}

To check non-convexity of a given fully reduced hourglass plabic graph, it suffices to consider \emph{minimal $\underline c_i$-cut paths} for each $i$, by which we mean $\underline c_i$-cut paths whose weight is minimal among the set of $\underline c_i$-cut paths.

\begin{remark}
    If $\alpha_1, \alpha_2, \alpha_3$ are the simple roots of $G$, then 
    \begin{align*}
        \alpha_1 &= 2\omega_1 - \omega_2,\\
        \alpha_2 &= -\omega_1 + 2\omega_2 - \omega_3,\\
        \alpha_3 &= -\omega_2+2\omega_3.
    \end{align*}
    As a consequence, \emph{minimal} cut paths can intersect \emph{at most one edge} incident to an internal vertex of simple degree $3$. To see this, let $v$ be an internal vertex of simple degree $3$, and let $\gamma$ be a cut path intersecting two edges incident to $v$. Then we can construct a new cut path $\gamma'$ by sliding $\gamma$ past $v$ to intersect only the other edge incident to $v$. Checking a few cases, we see $\weight(\gamma')\le \weight(\gamma)$, so $\gamma$ cannot be minimal.
\end{remark}

We first show that non-convexity is invariant under moves of an hourglass plabic graph.

\begin{lemma}\label{lem:nonconvex}
    Let $\underline C = (\underline c_1, \dots, \underline c_m)$ be a clasp sequence on $\underline a$. If $W, W'$ are move-equivalent fully reduced hourglass plabic graphs, then $W$ is non-convex if and only if $W'$ is non-convex.
\end{lemma}

\begin{proof}
    It suffices to show that (un)contraction moves, square moves, and benzene moves preserve non-convexity. Moreover, it suffices to consider minimal cut paths for each clasp.

    Suppose $u,v,w$ are vertices of $W$ such that $v$ is connected to both $u$ and $w$ via hourglass edges, and let $W'$ be the web obtained by contracting these two hourglass edges and collapsing $u,v,w$ to the single vertex $v$. Each cut path $\gamma'$ in $W'$ gives a corresponding cut path $\gamma$ in $W$ of the same weight, and hence non-convexity of $W$ implies that of $W'$. We wish to show that if $W'$ is non-convex, then so is $W$.
    
    If $v$ does not have simple degree $4$ in $W'$, then it is easy to see that every cut path in $W$ corresponds to one for $W'$ of the same weight, so we are done. The only non-trivial case is when $v$ has simple degree $4$ in $W'$.

    If $\gamma$ is a minimal $\underline c_i$-cut path in $W$ that does not intersect either of the hourglass edges joining, $v$ to $u$ or $w$, it once again corresponds to some cut path $\gamma'$ in $W'$ of the same weight and we are done. Hence, assume $\gamma$ intersects one of the hourglass edges. Sliding $\gamma$ past $u$ gives a new cut path $\gamma_1$, and similarly sliding past $w$ gives a new cut path $\gamma_2$. These paths correspond to cut paths $\gamma_1', \gamma_2'$ in $W'$ such that $\weight(\gamma_1) = \weight(\gamma_1')\ge \weight(\underline c_i), \weight(\gamma_2) = \weight(\gamma_2')\ge \weight(\underline c_i)$ by non-convexity of $W'$. But $\{\weight(\gamma_1),\weight(\gamma_2)\} = \{\weight(\gamma)+\alpha_1, \weight(\gamma) + \alpha_3\}$, and hence $\weight(\gamma) + \alpha_1\ge \weight(\underline c_i), \weight(\gamma)+ \alpha_3\ge \weight(\underline c_i)$. Rearranging gives $\weight(\gamma)-\weight(\underline c_i)\ge -\alpha_1, \weight(\gamma)-\weight(\underline c_i)\ge -\alpha_3$, and this is only possible if $\weight(\gamma)-\weight(\underline c_i)\ge 0$. Thus, $W$ is also non-convex. This shows that (un)contraction moves preserve non-convexity.
    \begin{figure}[h]
        \centering
        \includegraphics[width=0.4\linewidth]{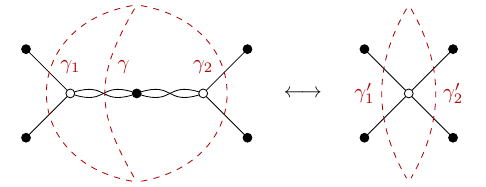}
        \label{fig:contraction_cutpaths}
    \end{figure}

    Next, suppose $F_0$ is a benzene face of $W$ and let $W'$ be obtained from $W$ via a benzene move at $F_0$. We wish to show that if $W$ is non-convex, then so is $W'$. Let $\gamma$ be a minimal $\underline c_i$-cut path in $W'$ and assume that $\gamma$ passes through $F_0$. We wish to show that $\weight(\gamma)\ge \weight(\underline c_i)$. Minimality of $\gamma$ lets us reduce to the case where $\gamma$ intersects exactly two edges of $F_0$.
    \begin{itemize}[leftmargin=*]
        \item If $\gamma$ intersects two adjacent edges of $F_0$, then we may ``slide" $\gamma$ past their common vertex to obtain a cut path of strictly smaller weight, contradicting minimality.

        \item If $\gamma$ intersects two edges of $F_0$ separated by an edge, then the contribution of these edges to the weight is either $\omega_1+\omega_3$ or $2\omega_2$, depending on whether these edges are simple or hourglasses, respectively. Comparing with the path $\gamma'$ obtained once again by ``sliding" $\gamma$ outside $F_0$, the two new intersections of $\gamma'$ contribute $\omega_1+\omega_3$. Therefore, $\weight(\gamma)\ge \weight(\gamma')$. Since $\gamma'$ has the same weight in $W$ and $W'$ (as it does not interact with the benzene face), non-convexity of $W$ implies that $\weight(\gamma) = \weight(\gamma')\ge \weight(\underline c_i)$.
        \item If $\gamma$ intersects two opposite edges of $F_0$, then one of these is a simple edge and the other is an hourglass edge. The same is true for $\gamma$ in $W$, so $\weight(\gamma)$ is the same whether calculated in $W$ or in $W'$. Non-convexity of $W$ then implies that $\weight(\gamma) \ge \weight(\underline c_i)$.
    \end{itemize}
    \begin{figure}[h]
        \centering
        \includegraphics[width=0.8\linewidth]{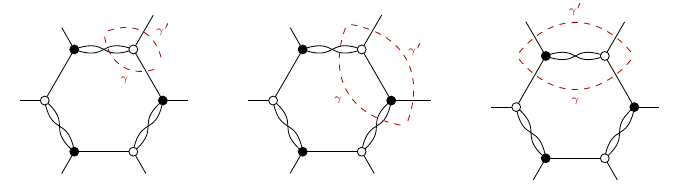}
        \label{fig:benzene_cutpaths}
    \end{figure}

    Lastly, suppose $F_0$ is a square face of $W$ and $W'$ is obtained from $W$ by applying a square move at $F_0$. By performing uncontraction moves if necessary, we may assume that the boundary vertices of the square are connected to hourglass edges in both $W$ and $W'$. We wish to show that if $W$ is non-convex and $\gamma$ is a $\underline c_i$-cut path in $W'$, then $\weight(\gamma)\ge \weight(\underline c_i)$. 
    
    Just as in the benzene face case, we need only consider minimal cut paths $\gamma$ which pass through $F_0$ in $W'$. If $\gamma$ intersects two adjacent edges of $F_0$, we may slide $\gamma$ past their intersection to obtain a cut path of strictly smaller weight, contradicting minimality. Thus $\gamma$ intersects two opposite edges of $F_0$ in $W'$, which contribute a weight of $\omega_1+\omega_3$. This gives a corresponding cut path in $W$ of the same weight as $\gamma$, and the non-convexity of $W$ then implies that $\weight(\gamma)\ge \weight(\underline c_i)$.
\end{proof}

The notion of non-convexity is useful in studying the kernel of the projection $\pi_{\underline C}$.

\begin{proposition}\label{prop:nonconvex}
    Suppose $W\in \CRG(\underline a)$ is partially convex with respect to $\underline C = (\underline c_1,\dots, \underline c_m)$. Then $[W]\in \ker\pi_{\underline C}$.
\end{proposition}

\begin{proof}
    Since $W$ is partially convex, there exists $1\le i\le m$ and a cut path $\gamma$ separating the $i$-th clasp from the rest such that $\weight(\gamma)\ngeq \weight(\underline c_i) = \lambda_i$. For simplicity, let $M = \bigotimes_{1\le j\le m, j\ne i}V^{\otimes \underline c_j}$. Then $V^{\otimes \underline a}\simeq V^{\otimes \underline c_i}\otimes M$, so that
    \[\inv_G(V^{\otimes \underline a})\simeq \hom_G\left(V^{\otimes \underline c_i}, M^*\right).\]
    Hence $[W]$ can be thought of as a $G$-invariant map from $V^{\otimes \underline c_i}$ to $M^*$. 

    Let $n_1, n_2, n_3$ be as in the definition of $\weight(\gamma)$, and let 
    \[V_\gamma = V(\omega_1)^{\otimes n_1}\otimes V(\omega_2)^{\otimes n_2}\otimes V(\omega_3)^{\otimes n_3}.\] 
    Then $\gamma$ splits $W$ into two webs $W_1, W_2$ (as in \Cref{fig:web_cut}) such that $[W] = [W_2]\o[W_1]$ with $[W_1]\in \hom_G(V^{\otimes \underline c_i}, V_\gamma)$ and $[W_2]\in\hom_G(V_\gamma, M^*)$.

    We consider the restriction of $[W_1]$ to the irreducible component $V(\lambda_i)\subset V^{\otimes\underline c_i}$. Every irreducible component $V(\mu)\subset V_\gamma$ has highest weight $\mu\le n_1\omega_1 + n_2\omega_2 + n_3\omega_3 =\weight(\gamma)$. By hypothesis, $\lambda_i\nleq \weight(\gamma)$ so $V(\lambda_i)$ does not appear as a irreducible factor of $V_\gamma$. This implies that $[W_1]\Big|_{V(\lambda_i)} = 0$, and hence $[W]\Big|_{V(\lambda_i)} = 0$. In other words, $[W]\in \ker\pi_{\underline C}$.

    \begin{figure*}[h]
        \centering
        \includegraphics[height = 5cm]{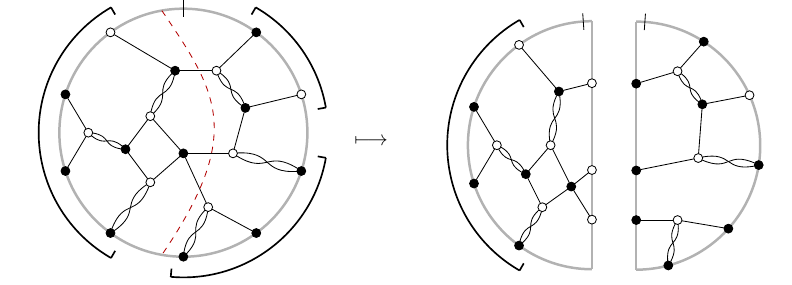}
        \caption{A cut path splitting a web into two.}
        \label{fig:web_cut}
    \end{figure*}
\end{proof}

\subsection{Bad local configurations}

\Cref{fig:bad_configs} gives a list of local boundary configurations which are partially convex, along with the cut path witnessing this partial convexity. We say that $W\in \CRG(\underline a)$ has a \emph{bad local boundary configuration within a clasp} if it contains one of the configurations from \Cref{fig:bad_configs} as a subgraph, such that both boundary vertices belong to the same clasp.

\begin{figure}[h]
\centering
\includegraphics[width=15cm]{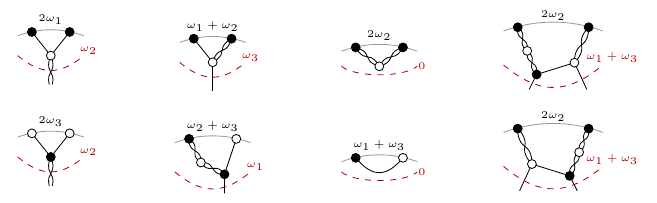}
\caption{The partially convex local boundary configurations, with partial convexity witnessed by the dashed arc.}
\label{fig:bad_configs}
\end{figure}

\section{Descents and bad local boundary configurations}
\label{sec:descents}

We wish to characterize the non-convexity of a clasped web $W$ in terms of of its lattice word $\L(W)$. For this purpose, we introduce the notion of \emph{descents}.

\begin{definition}\label{def:descent}
A sorted word $L = w_1\cdots w_n$ has a \emph{descent} at position $i$ if one of the following holds:
\begin{enumerate}
    \item $w_i = a\in [r], w_{i+1} = b\in [r]$, such that $a < b$.
    \item $w_i = \overline a, w_{i+1} = \overline b$, such that $a>b$.
    \item $w_i = 1, w_{i+1} = \overline 1$.
    \item $w_i = a\in [r], w_{i+1} = \{b,c\}\in {[r]\choose 2}$, such that $a< b< c$.
    \item $w_i = \{a,b\}, w_{i+1} = \overline c$, such that $c < \min\{[r]\setminus \{a,b\}\}$.
    \item $w_i = \{a,b\}, w_{i+1} = \{c,d\}$, such that $a< c$ or $b < d$.
\end{enumerate}
If $L$ has a descent at position $i$ we write $w_i \ngeq w_{i+1}$, and if not we write $w_i\ge w_{i+1}$.
\end{definition}

\begin{remark}
    Treating barred letters as elements of ${[r]\choose 3}$ (by taking complements) makes the notion of descent a bit more transparent. If $A \in {[r]\choose i}$ and $B \in {[r]\choose j}$ with $i\le j$, then $A\ge B$ if and only if the $k$-th element of $A$ is at least as large as the $k$-th element of $B$ for all $k\le i$ (cf. \emph{Gale} order).
\end{remark}

We are now able to define a condition on lattice words which we will later show to characterize non-convexity.

\begin{definition}\label{def:latticeword}
    Given a sequence of sorted clasps $\underline C = (\underline c_1,\dots, \underline c_m)$, a lattice word $L$ has no \emph{$\underline C$-descents} if $L = L_1\cdots L_m$ with $\type(L_i) = \underline c_i, \forall i$ and each $L_i$ has no descents.
    The set of all lattice words with no $\underline C$-descents is denoted $\mathsf L(\underline C)$, and the set of all balanced lattice words with no $\underline C$-descents is denoted $\BL(\underline C)$.
\end{definition}

Our immediate goal is to show that each descent in the lattice word $\L(W)$ gives rise to a bad local boundary configuration in $W$. As a corollary, $W$ being non-convex will imply that $\L(W)$ has no $\underline C$-descents. The main tool we use in this section is the growth algorithm of \cite{GPPSS-sl4}, which gives a way of reconstructing an element of $\CRG(\underline a)\slash\sim$ from a balanced lattice word $L\in \BL(\underline a)$.

\begin{algorithm} (Growth algorithm) \label{algo:growth} Let $L\in \BL(\underline a)$.
\begin{enumerate}
    \item Oscillize to obtain $\osc(L)$. Draw a horizontal line with downward dangling strands labelled by letters of $\osc(L)$, oriented upwards for barred letters and downwards for unbarred letters.
    \item As long as any dangling strands remain, apply the growth rules from \Cref{fig:growth_rules} to obtain a symmetrized six-vertex configuration.
    \item Convert to an hourglass plabic graph and de-oscillize.
\end{enumerate}
\end{algorithm}

\begin{figure}
    \centering
    \includegraphics[width=\linewidth]{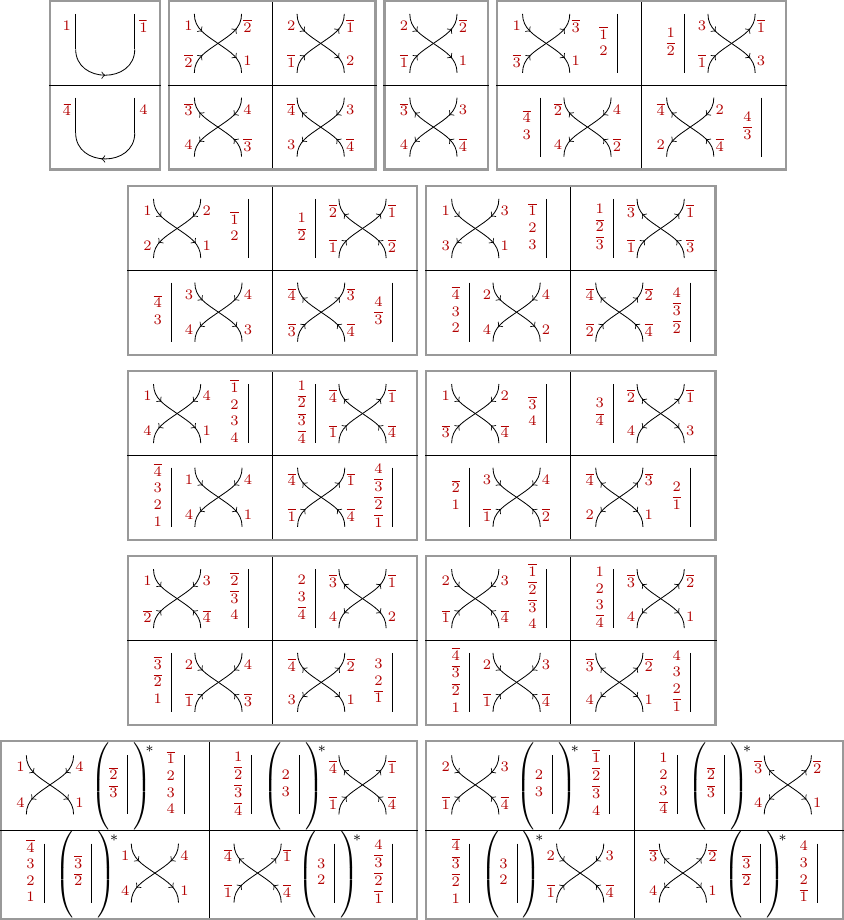}
    \caption{Growth rules for $\SL_4$ webs. A vertical line indicates that a witness with one of the labels must be present. Vertical lines with a $*$ indicate that any number of such witnesses (including zero) can be present.}
    \label{fig:growth_rules}
\end{figure}

The following theorem is the main result of \cite[\S 5]{GPPSS-sl4}.

\begin{thm}\label{thm:growth}
    Given $L\in \BL(\underline a)$, the growth algorithm always terminates in finitely many steps, resulting in a fully reduced hourglass plabic graph $W$. The move-equivalence class $W$ is independent of the order in which growth rules are applied, and thus the growth algorithm gives a well defined map
    \[\mathcal G:\BL(\underline a)\to \CRG(\underline a)\slash\sim.\]
    Moreover, $G$ is a bijection with inverse $\CRG(\underline a)\slash\sim\ \to \BL(\underline a)$ given by $W\mapsto \mathcal L(W)$.
\end{thm}

We also record a useful lemma from \cite{GPPSS-sl4} about descents in the oscillating case:

\begin{lemma}\label{lem:descent}
    Let $L = w_1\cdots w_n\in \BL(\underline a)$, and $W = \mathcal G(L)$ with corresponding boundary vertices $b_1,\dots, b_n$. Suppose $|w_i|= |w_{i+1}|\in \{1, 3\}$. Then $w_i\ngeq w_{i+1}$ if and only if $b_i$ and $b_{i+1}$ are connected to the same vertex in $W$.
\end{lemma}

We use growth rules now to show that descents imply bad local boundary configurations:

\begin{proposition}
    \label{prop:descent}
    Let $W\in \CRG(\underline a)$ with adjacent boundary vertices $b_i, b_{i+1}$, and let $L = \mathcal L(W) = w_1\cdots w_n$. If $w_i\ngeq w_{i+1}$, then $W$ has a bad local boundary configuration between $b_i$ and $b_{i+1}$.
\end{proposition}

\begin{proof}
    First, it is an easy check that bad local boundary configurations are preserved by benzene moves and square moves, so this is really a statement about $\CRG(\underline a)\slash\sim$.

    Assuming $w_i\ngeq w_{i+1}$, we will show using \Cref{algo:growth} that $\mathcal G(L)$ has a bad local boundary configuration between $b_i$ and $b_{i+1}$. Since $W\sim \mathcal G(L)$ by \Cref{thm:growth}, it follows that $W$ has a bad local boundary configuration between $b_i$ and $b_{i+1}$. The fact that $w_i\ngeq w_{i+1}$ implies one of the following 6 possibilities:

    \begin{enumerate}[leftmargin=*]
        \item $w_i = a\in [4], w_{i+1}=b\in [r]$ with $a<b$. In this case, \Cref{lem:descent} implies that $b_i$ and $b_{i+1}$ have a common neighbor in $W$, which is a bad local boundary configuration. 

        \item $w_i = \overline a\in \overline {[4]}, w_{i+1} = \overline b\in \overline{[4]}$ with $a>b$. Once again, \Cref{lem:descent} implies that $b_i$ and $b_{i+1}$ have a common neighbor.

        \item $w_i = 1, w_{i+1} = \overline 1$. The growth rule $1\overline 1\to \emptyset$ applies, showing that $b_i$ and $b_{i+1}$ are connected. This is a bad local boundary configuration.

        \item $w_i = a\in [4], w_{i+1} = \{b,c\}\in {[4]\choose 2}$ such that $a<b<c$. In $\osc_i(W)$, $b_{i+1}$ splits into two boundary vertices $b_{i+1}$ and $b_{i+1}'$. Since $a<b$, \Cref{lem:descent} implies that $b_i$ is adjacent to $b_{i+1}$ in $\osc_i(W)$, and hence de-oscillizing gives a bad local boundary configuration.

        \item $w_i = \{a,b\}\in {[4]\choose 2}, w_{i+1} = \overline c\in \overline{[4]}$ such that $c< \min
        \left\{[4]\setminus\{a,b\}\right\}$. For ease of notation, let $[4]\setminus \{a,b\} = \{d, e\}$ with $d<e$. Before oscillizing, we replace $w_i$ by $\overline{[4]\setminus\{a,b\}} = \{\overline d, \overline e\}\in {\overline{[4]}\choose 2}$, which on the level of webs corresponds to applying $\flip_i$. Then the descent condition becomes $c<d< e$. 
        
        In $\osc_i(W)$, $b_i$ splits into two boundary white vertices $b_i, b_i'$. The fact that $c<d$ then implies by \Cref{lem:descent} that $b_i'$ and $b_{i+i}$ have a common neighbor. De-oscillizing implies that $b_i$ and $b_{i+1}$ have a common neighbor, where $b_i$ is a white vertex incident to an hourglass edge. Applying $\flip_i$ again gives a bad local boundary configuration.

        \item $w_i = \{a,b\}, w_{i+1} = \{c,d\}$ such that $a<c$ or $b<d$. We list out all possible cases and examine possible growth rules. The possibilites for $\{a,b\}\{c,d\}$ are

        \begin{tabular}{ccccc}
             $\{1, 2\}\{1, 3\}$ & $\{1, 2\}\{1, 4\}$ & $\{1, 2\}\{2, 3\}$ & $\{1, 2\}\{2, 4\}$ & $\{1, 2\}\{3, 4\}$\\
             $\{1, 3\}\{1, 4\}$ & $\{1, 3\}\{2, 3\}$ & $\{1, 3\}\{2, 4\}$ & $\{1, 3\}\{3, 4\}$\\
             $\{1, 4\}\{2, 3\}$ & $\{1, 4\}\{2, 4\}$ & $\{1, 4\}\{3, 4\}$\\
             $\{2, 3\}\{1, 4\}$ & $\{2, 3\}\{2, 4\}$ & $\{2, 3\}\{3, 4\}$\\
             $\{2, 4\}\{3, 4\}$
        \end{tabular}
        \vspace{3mm}

        We will show that oscillizing, applying the growth algorithm, and then de-oscillizing always produces a bad local boundary configuration. In all cases except $\{1, 2\}\{1, 3\}$ and $\{2, 4\}\{3, 4\}$, \Cref{fig:descents} shows the growth rules to be applied in order to create a bad local boundary configuration (see \Cref{fig:desc_badconfig}).

        \begin{figure}[h]
            \centering
            \includegraphics[width=0.8\linewidth]{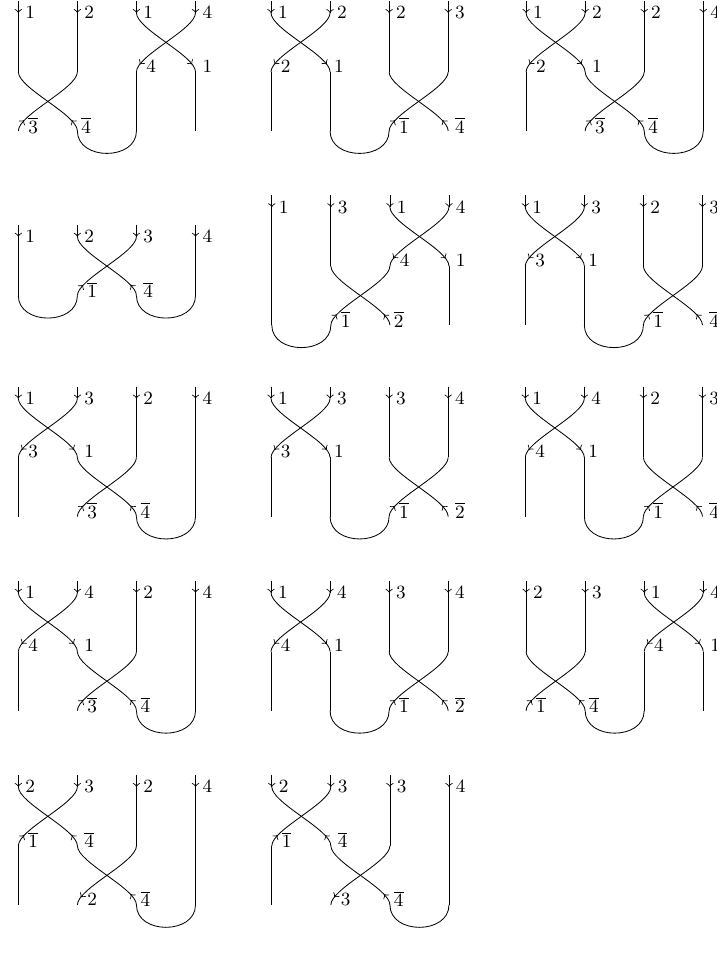}
            \caption{Growth rules to be applied in all cases except $\{1,2\}\{1,3\}$ and $\{2, 3\}\{2, 4\}$.}
            \label{fig:descents}
        \end{figure}

        \begin{figure}[h]
            \centering
            \includegraphics[width=0.7\linewidth]{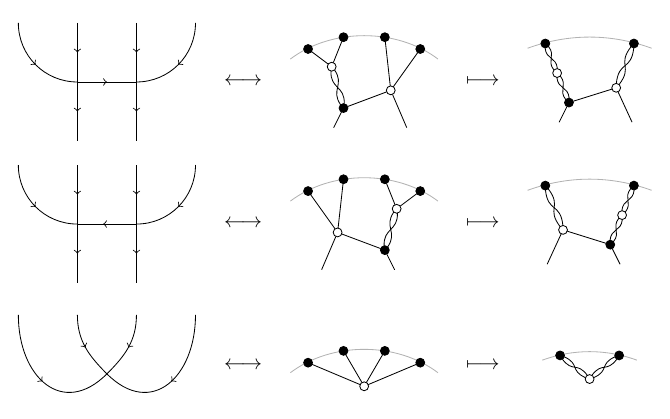}
            \caption{Bad local boundary configurations corresponding to the symmetrized six-vertex configurations in \Cref{fig:descents}.}
            \label{fig:desc_badconfig}
        \end{figure}

        In the case of $\{1,2\}\{1,3\}$, no growth rules can be applied immediately, without the existence of certain witnesses. If $3$ has a witness of $\overline1, 2$ or $3$ following it, we may apply the $13\to 31$ growth rule and form a bad local boundary configuration as shown in \Cref{fig:1213}. Similarly, if $3$ has a witness of $\overline 2,\overline 3$ or $4$ following it, we may apply the $13\to \overline2\overline4$ growth rule to obtain a bad local boundary configuration (see \Cref{fig:1213}).

        \begin{figure}[h]
            \centering
            \includegraphics[width=0.7\linewidth]{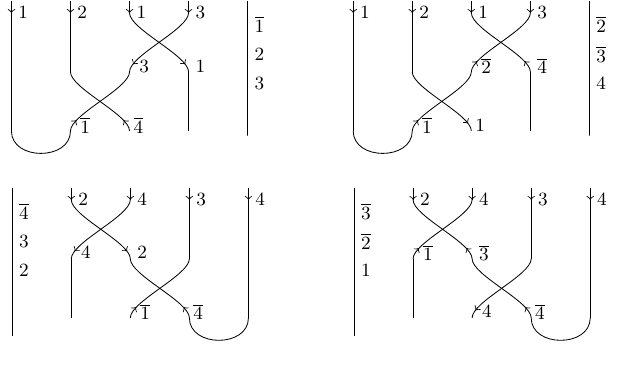}
            \caption{Growth rules applied in the case of $\{1, 2\}\{1, 3\}$ (top row) and $\{2, 3\}\{2, 4\}$ (bottom row), in the presence of certain witnesses.}
            \label{fig:1213}
        \end{figure}

        The only remaining case is if $3$ does not have any of these witnesses following it. Then we observe that no growth rules are applicable involving any of these four strands (note that no growth rule can affect the leading $1$). For the growth algorithm to eventually terminate, some sequence of growth rules occurring beyond the $1213$ string must result in one of the above witnesses following $3$, reducing to the case of the previous paragraph.

        The case of $\{2,4\}\{3, 4\}$ is dealt with in exactly the same way, by observing the witnesses preceding $2$ to apply either $24\to 42$ or $24\to \overline 1\overline 3$ (see \Cref{fig:1213}).
    \end{enumerate}
\end{proof}

\section{Lattice words and Littlewood--Richardson tableaux}\label{Sec:Fund}

The goal of this section is to show that the number of sorted balanced lattice words with no $\underline C$-descents (see \Cref{def:latticeword}) equals $\dim\inv_G(\otimes_{i=1}^m V(\lambda_i))$. We first formulate another equivalent criterion for checking when a lattice word $L$ has no descents.

\begin{definition}
    Let $w \in [4]\sqcup {[4]\choose 2}\sqcup \overline{[4]}$ be a letter, and for $i\in \{1,2,3\}$ such that $|w|\ge i$, define $w^{(i)}$ to be the $i$-th smallest element of the set $w$ (treating barred letters as elements of ${[4]\choose 3}$ by taking complements). We extend this to lattice words $L = w_1\cdots w_m$ by setting $L^{(i)} = w_1^{(i)}\cdots w_m^{(i)}$, where we omit letters $w_j^{(i)}$ that are not defined (i.e. when $|w_j|<i$).
\end{definition}

\begin{example}
    If $L = 11232\{1,4\}\{2,3\}\{1,3\}\overline 3 \overline 2\overline 2\overline 4$, then
    \begin{align*}
    L^{(1)} &= 112321211111,\\
    L^{(2)} &= 4332332,\\
    L^{(3)} &= 4443.\\
    \end{align*}
\end{example}

\begin{lemma}
\label{lem:descent2}
    Let $v,w$ be adjacent letters in a sorted word $L$. Then $v\ge w$ if and only if $v^{(i)}\ge w^{(i)}$ for all $i\le |v|\le |w|$. Therefore, $L$ has no descents if and only if each $L^{(i)}$ is a word with weakly decreasing letters.
\end{lemma}

\begin{proof}
    We work through all possible cases for $|v|$ and $|w|$:
    \begin{itemize}[leftmargin=*]
        \item $|v| = |w| = 1$: In this case $i = 1$, $v^{(1)} = v, w^{(1)} = w$ so there is nothing to prove.
        \item $|v| = 1, |w| = 2$: Once again $i = 1$. Let $w = \{a,b\}$ with $a<b$. Then $v\ngeq w$ if and only if $v < a$, i.e., if and only if $v^{(1)} < w^{(1)}$. Thus $v\ge w$ if and only if $v^{(1)}\ge w^{(1)}$.
        \item $|v| = 1, |w| = 3$: Again $i = 1$. Let $w = \overline a$, and $[4]\setminus \{a\} = \{b,c,d\}$ with $b< c< d\le 4$. But this implies that $b\le 2$. Thus $v^{(1)} \ge w^{(1)} = b$ unless $b = 2$ and $v = 1$. This is equivalent to saying $v = 1, w = \overline 1$, i.e. that $v\ngeq w$.
        \item $|v| = 2, |w| = 2$: Here $i = 1$ or $2$. Let $v = \{a,b\}, w = \{c,d\}$. Then $v\ngeq w$ if and only if $a<c$ or $b<d$, i.e. if and only if $v^{(1)}< w^{(1)}$ or $v^{(2)}< w^{(2)}$. Negating both statements says that $v\geq w$ if and only if $v^{(1)}\ge w^{(1)}$ and $v^{(2)}\ge w^{(2)}$.
        \item $|v| = 2, |w| = 3$: Here $i = 1$ or $2$. Let $v = \{a,b\}, w = \overline c$, with $a<b$. Then $v\ngeq w$ if and only if $c<\min \{[4]\setminus \{a,b\}\}$, and this is possible only if $c=1$ or $2$. In the former case, we must have $a = 1$ (if not, $1\in [4]\setminus \{a,b\}$) and so $v^{(1)} = 1< 2 = w^{(1)}$. In the latter case, we have $c = 2 < \min\{[4]\setminus\{a,b\}\}$ which implies that $v = \{a,b\} = \{1,2\}$. Thus $v^{(2)} = 2 < 3 = w^{(2)}$. Thus $v\ngeq w \implies v^{(1)}< w^{(1)}$ or $v^{(2)}< w^{(2)}$. 
        
        On the other hand suppose $v\ge w$, so that $c \ge \min\{[4]\setminus \{a,b\}\}$. If $c\ne a$, then $a \in [4]\setminus \{c\}$ so that $w^{(1)} = \min\{[4]\setminus \{c\}\}\le a = v^{(1)}$. If $c = a$, then $w^{(1)} = \min\{[4]\setminus \{c\}\} \le \min\{[4]\setminus\{a,b\}\}\le c = a = v^{(1)}$. In either case $v^{(1)}\ge w^{(1)}$. 

        Since $w^{(2)}$ is the second largest element of $[4]\setminus \{c\}$, we see that $w^{(2)}\le 3 $. So if $b \ge 3$, then $b = v^{(2)}\ge w^{(2)}$. The only remaining case is if $b = 2$ (note that $b\ne 1$ since $b>a$), and hence $a = 1$. This implies that $c\ge 3 = \min\{[4]\setminus \{a,b\}\}$, and hence $[4]\setminus\{c\} = \{1,2,3\}$ or $\{1,2,4\}$. In either case, $w^{(2)} = 2$ and hence $v^{(2)}\ge w^{(2)}$. This shows that $v\ge w \implies v^{(1)}\ge w^{(1)}$ and $v^{(2)}\ge w^{(2)}$.

        \item $|v| = |w| = 3$: Here $i = 1,2$ or $3$. Let $v = \overline a, w = \overline b$. Recall that $v\ngeq w$ if and only if $a> b$. If $a> b$ then $v^{(b)} = b < b+1 = w^{(b)}$. If $a\le  b$, then $v^{(i)} = i = w^{(i)}$ if $i< a$, $v^{(i)} = i+1 = w^{(i)}$ if $b\le i \le 3$, and $v^{(i)} = i+1 > i = w^{(i)}$ if $a\le i< b$. This shows that $v^{(i)}\ge w^{(i)}$ for all $i\in[3]$. Therefore $a> b$ if and only if $v^{(i)}< w^{(i)}$ for some $i\in[3]$.
    \end{itemize}
     
\end{proof}

\begin{definition}
    Given a sequence of weights $\underline \lambda = (\lambda_1,\dots, \lambda_m)$, a \emph{$\underline \lambda$-Littlewood--Richardson tableau} $T$ is a sequence of nested partitions $\emptyset=\mu_0\subset \mu_1\subset\cdots \subset \mu_m$ along with a Littlewood--Richardson filling of each skew shape $\mu_i\setminus \mu_{i-1}$ with content $\lambda_i$, i.e., each $\mu_i\setminus\mu_{i-1}$ has a semistandard filling of content $\lambda_i$ such that the reading word (read from right to left, top to bottom) is a lattice word. We say that the shape of $T$ is $\mathsf{sh}(T) = \mu_m$. The set of all $\underline \lambda$-Littlewood--Richardson tableaux is denoted by $\mathsf T(\underline \lambda)$, and the set of all $\underline \lambda$-Littlewood--Richardson tableaux of rectangular shape with $r$-rows is denoted by $\RT(\underline\lambda)$.
\end{definition}

\begin{remark}
    Recall that a weight $\lambda_i$ gives a partition by the rule $\omega_k\mapsto (1^k)$, which is how we can interpret the content of a filling as a weight.
\end{remark}

\begin{example}
\Cref{table:LR} gives an example of a $\underline\lambda$-Littlewood--Richardson tableau of shape $(5,5,5,5,5)$, where $\underline\lambda = (\lambda_1, \lambda_2, \lambda_3)$ with 
\[\lambda_1 = 2\omega_1+\omega_2 = (3, 1, 0, 0),\]
\[\lambda_2 = 2\omega_1+3\omega_3 = (5, 3, 3, 0),\]
\[\lambda_3 = 2\omega_1+\omega_3 = (3, 1, 1, 0).\]
The chain of partitions $\emptyset=\mu_0\subset \mu_1\subset \mu_2\subset\mu_3$ in the example is 
\[\mu_1 = (3, 1, 0, 0),\]
\[\mu_2 = (5, 4, 4, 2),\]
\[\mu_3 = (5, 5, 5, 5).\]

\begin{table*}[h]
    \begin{tabular}[scale=1.5]{|c|c|c|c|c|}
    \hline
        \cellcolor{blue!40}$1$ & \cellcolor{blue!40}$1$ & \cellcolor{blue!40}$1$ & \cellcolor{blue!20}$1$ & \cellcolor{blue!20}$1$ \\
    \hline
        \cellcolor{blue!40}$2$ & \cellcolor{blue!20}$1$ & \cellcolor{blue!20}$1$ & \cellcolor{blue!20}$2$ & \cellcolor{blue!10}$1$ \\
    \hline 
        \cellcolor{blue!20}$1$ & \cellcolor{blue!20}$2$ & \cellcolor{blue!20}$2$ & \cellcolor{blue!20}$3$ & \cellcolor{blue!10}$2$ \\
    \hline 
        \cellcolor{blue!20}$3$ & \cellcolor{blue!20}$3$ & \cellcolor{blue!10}$1$ & \cellcolor{blue!10}$1$ & \cellcolor{blue!10}$3$ \\
    \hline
    \end{tabular}
    \caption{Example of a rectangular Littlewood--Richardson tableau.}
    \label{table:LR}
\end{table*}
\end{example} 

\begin{remark}
    \label{rem:schur}
    For any sequence of weights $\underline \lambda = (\lambda_1,\dots, \lambda_m)$, the Littlewood--Richardson rule implies that the coefficient of $s_\mu$ when $s_{\lambda_1}\cdots s_{\lambda_m}$ when expressed in the Schur polynomial basis is equal to the number of $\underline \lambda$-Littlewood--Richardson tableaux of shape $\mu$. In particular, if $\lambda_i=\weight(\underline c_i)$, then
    \[|\RT(\underline C)| = \dim \inv_G\left(\bigotimes_{i=1}^m V(\lambda_i)\right).\]
\end{remark}

\begin{proposition}\label{prop:LR_bijection}
    Fix a sequence of sorted clasps $\underline C = (\underline c_1,\dots, \underline c_m)$, and let $\weight(\underline C) = (\weight(\underline c_1),\dots, \weight(\underline c_m))$. Then there is a bijection $\BL(\underline C)\simeq \RT(\weight(\underline C))$.
\end{proposition}

\begin{proof}
    Suppose first that we are given $L\in \BL(\underline C)$, so $L= L_1\cdots L_m$ such that for each $i$, $L_i$ is of type $\underline c_i$ and has no descents. Assume that $\lambda_1\subset\cdots\subset \lambda_{i-1}$ have been constructed along with Littlewood--Richardson fillings. We shall use the subword $L_i$ to construct $\lambda_i$.

    Picturing $\lambda_{i-1}$ as a Young diagram, we will construct $\lambda_i$ by adding boxes and filling them in as we go, according to the letters in $L_i$ read from left to right as follows:
    \begin{itemize}[leftmargin=*]
        \item For a letter of the form $a\in [4]$, add a box to row $a$ and fill it in with the number $1$.
        \item For a letter of the form $\{a,b\}\in{[4]\choose 2}$ with $a<b$, add a box filled with $1$ to row $a$ and a box filled with $2$ to row $b$.
        \item For a letter of the form $\overline a$, let $\{b,c,d\} = [4]\setminus \{a\}$ with $b<c<d$, and add boxes in rows $b, c, d$ filled with $1,2,3$ respectively.
    \end{itemize}
    Since $L$ is a lattice word, adding any boxes according to the above recipe still gives a partition, and hence we get the partition $\lambda_i$ with a filling of the skew shape $\lambda_i\setminus \lambda_{i-1}$. The reading word of the filling is a lattice word, because every $2$ added is added with a $1$ above it, and every $3$ added is added with a $2$ above it. All that is left to check is that the filling is semistandard, but this is implied by the fact that $L_i$ has no descents, so the rows to which $1$s are added (when reading $L_i$ from left to right) are weakly decreasing, and similarly for $2$ and $3$. By construction, the filling of $\lambda_i\setminus\lambda_{i-1}$ has content $\weight(\underline c_i)$, and the resulting Littlewood--Richardson tableau has shape $\lambda_m$ a rectangle because $L$ is balanced.

    Conversely given $T\in \RT(\weight(\underline C))$, we use the Littlewood--Richardson fillings $\lambda_i\setminus\lambda_{i-1}$ to construct the subwords $L_i$, which we concatenate to get $L = L_1\cdots L_m$. We construct $L_i$ letter by letter from right to left using the filling of $\lambda_i\setminus \lambda_{i-1}$ as follows:
    \begin{enumerate}[leftmargin=*]
        \item If there is no $3$ in the filling, skip this step. Find the smallest $c$ such that a $3$ appears in row $c$. Since this is a Littlewood--Richardson filling, there is at least one $2$ strictly above every $3$. Find the smallest $b$ such that a $2$ appears in row $b$, so $b<c$. Finally, find the smallest $a$ such that a $1$ appears in row $a$, so $a<b<c$. If $\{d\} = [4]\setminus \{a,b,c\}$, then add the letter $\overline d$ to $L_i$ and delete the rightmost $1,2,3$ from rows $a,b,c$ respectively. Repeat this step as long as $3$s remain.
        \item If there is no $2$ in the filling, skip this step. Find the smallest $b$ such that a $2$ appears in row $b$, and the smallest $a$ such that a $1$ appears in row $a$. Then $a<b$. Add the letter $\{a,b\}$ to $L_i$ and delete the rightmost $1,2$ from rows $a,b$ respectively. Repeat this step as long as $2$s remain.
        \item Find the smallest $a$ such that a $1$ appears in row $a$. Add the letter $a$ to $L_i$ and delete the rightmost $1$ in row $a$. Repeat this step as long as $1$s remain.
    \end{enumerate}

    By construction, $L_i$ is of type $\underline c_i$. Next we check that $L_i$ consists of no descents. For each $j\in\{1,2,3\}$, $L_i^{(j)}$ read from right to left is the sequence rows in $\lambda_i\setminus \lambda_{i-1}$ from which $j$'s were removed. By construction, this is a weakly increasing sequence, and hence $L_i^{(j)}$ read from left to right is a weakly decreasing sequence. By \Cref{lem:descent2}, this implies that $L_i$ has no descents.

    Lastly, since at each stage of the process of removing boxes we are still left with a Young diagram, the resulting word $L = L_1\cdots L_m$ is a lattice word, and is balanced since $T$ is rectangular. It is a routine check that these two operations give inverse bijections between $\BL(\underline C)$ and $\RT(\weight(\underline C))$.
\end{proof}

Coupled with \Cref{rem:schur}, we deduce that $\BL(\underline C)$ counts the dimension of the corresponding invariant space.

\begin{corollary}\label{cor:LR}
    For a sequence of sorted clasps $\underline C = (\underline c_1,\dots, \underline c_m)$, $$|\BL(\underline C)| = \dim\inv_G\left(\bigotimes_{i=1}^m V(\lambda_i)\right).$$
\end{corollary}

\section{Main theorem for sorted clasps}
\label{sec:main-sorted-proof}

We are now ready to prove the theorem for sorted clasps.
\begin{thm}\label{thm:sorted}
    Fix boundary conditions $\underline a \in [3]^n$ and a \emph{sorted} clasp sequence $\underline C = (\underline c_1,\dots, \underline c_m)$, and let $\lambda_i = \weight(\underline c_i)$. Then a basis for $\inv_G\left(\bigotimes_{i=1}^m V(\lambda_i)\right)$ is given by
    \[\mathcal W_{\underline C} \coloneqq \{\pi_{\underline C}([W]) \mid [W]\in \mathcal W_{\underline a}\setminus \ker\pi_{\underline C}\}.\]
    Moreover, the following are equivalent for $W\in \mathcal W_{\underline a}$:
    \begin{enumerate}
        \item $[W] \notin \ker\pi_{\underline C}$.
        \item $W$ is non-convex.
        \item There are no local boundary configurations from \Cref{fig:bad_configs} occurring within any clasp of $W$.
        \item The lattice word $\L(W) = \partial\sep_W$ has no $\underline C$-descents.
        \item $W$ has no trips that start and end in the same clasp.
    \end{enumerate}

    \begin{proof}
        From \Cref{prop:nonconvex}, we know that $\mathcal W_{\underline C}$ as defined in the theorem spans $\inv_G(\bigotimes_{i=1}^m V(\lambda_i))$, so the first claim follows if we show that $|\mathcal W_{\underline C}|\le \dim\inv_G (\bigotimes_{i=1}^m V(\lambda_i))$. We will show this by simultaneously proving the equivalence of (1), (2), (3), and (4).
        
        (1)$\implies$(2) is \Cref{prop:nonconvex},(2)$\implies$(3) is seen directly from \Cref{fig:bad_configs}, and (3)$\implies$(4) is \Cref{prop:descent}. Using (1)$\implies$(4) with the fact that webs with the same lattice word are move equivalent, we deduce that 
        \[|\mathcal W_{\underline C}|\le |\BL(\underline C)| = \dim \inv_G\left(\bigotimes_{i=1}^m V(\lambda_i)\right),\]
        where the last equality follows from \Cref{cor:LR}. This shows that $\mathcal W_{\underline C}$ is indeed a basis for $\inv_G\left(\bigotimes_{i=1}^m V(\lambda_i)\right)$. As a result $|\mathcal W_{\underline C}| = |\BL(\underline C)|$, and therefore (4)$\implies$(1). This shows the equivalence of (1), (2), (3) and (4).

        It is easy to check (see \Cref{fig:badconf_trips}) that every bad local configuration has a trip that starts and ends in the same clasp, so (5)$\implies$(3). It now suffices to check that if $W$ is non-convex, then it satisfies (5).

        \begin{figure}[h]
            \centering
            \includegraphics[width=15cm]{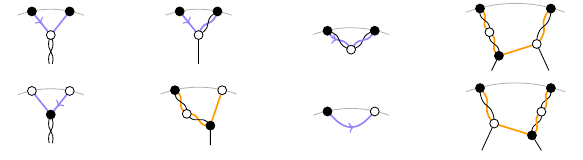}
            \caption{Returning $\trip_1$- (in purple) and $\trip_2$-strands (in orange) within the bad local boundary configurations from \Cref{fig:bad_configs}.}
            \label{fig:badconf_trips}
        \end{figure}

        By the rotation invariance of non-convexity and property (5), it suffices to show that if $W$ is non-convex, no trip that starts in the \textit{first} clasp can end in the \textit{first} clasp. Assuming $W$ is non-convex, the equivalence of (2) and (4) implies that $\L(W) = L_1\cdots L_m$ where $L_1$ has no descents. Since $L$ is a lattice word, so is $L_1$. This forces $L_1$ to be of the form
        \[L_1 = 1\cdots 1\{1, 2\}\cdots \{1,2\}\overline 4\cdots \overline 4.\] Assume for the sake of contradiction that there is a trip $\ell$ starting at $b_i$ and ending at $b_j$ within the first clasp, with $i< j$. Let $e_i, e_j, F(e_i), F(e_j)$ be as in the definition of $\sep_W$.
        \begin{itemize}[leftmargin=*]
            \item Suppose $a_j = 1$, so $\sep_W(e_j) = 1$. This implies that no trip through $e_j$ separates $F(e_j)$ from the base face $F_0$, contradicting the fact that $\ell$ is such a trip.

            \item Suppose $a_j = 2$. We may partially oscillize $W$ to $W' = \osc_j(W)$ by splitting $b_j$ into two distinct black vertices $b_j, b_j'$. Then $\sep_{W'}(e_j) = 1, \sep_{W'}(e_j') = 2$. The fact that $\sep_W(e_j) = 1$ implies that no trip through $e_j$ separates $F(e_j)$ from the base face $F_0$, so $\ell$ must end at $b_j'$ and not $b_j$. Note that $F(e_j')$ is the face bounded by $e_j, e_j',$ and the boundary of the disk. One trip through $e_j'$ separates $F(e_j')$ from $F_0$ since $\sep_W(e_j')= 2$. But the $\trip_1$-strand through $b_j$ traversing exactly $e_j, e_j'$ is such a trip, contradicting the fact that $\ell$ is another such trip.

            \item Suppose $a_j = 3$, so $\sep_W(e_j) = 4$. This implies that every trip through $e_j$ separates $F(e_j)$ from $F_0$, contradicting the fact that $\ell$ is a trip that does not do so.
        \end{itemize}

        In each case we arrive at a contradiction, and hence there is no trip starting and ending at the first clasp, showing that (2)$\implies$(5) as desired.
    \end{proof}
\end{thm}

\section{Sorting boundary conditions}
\label{sec:sorting}

The goal of this section is to construct a map $\swap$ that swaps two adjacent boundary conditions of a fully reduced hourglass plabic graph. More precisely, let $\underline a = (a_1,\dots, a_n)$ such that $a_1\ne a_2$. Define $\underline a' = (a_1', \dots, a_n')$ by
\[a_i' = \begin{cases}
    a_2 & \text{if }i=1,\\
    a_1 & \text{if }i=2,\\
    a_i &\text{otherwise.}
\end{cases}\]
Then we will construct a map $\swap:\CRG(\underline a)\slash\sim\ \to \CRG(\underline a')\slash\sim$ that preserves non-convexity and returning trips (with respect to some chosen clasp sequence). Conjugating by a suitable rotation gives the map $\swap_{i, i+1}$ that swaps the boundary conditions $a_i, a_{i+1}$ for any $i\in[n]$.

For simplicity, we first consider the \textit{oscillating} case; the general case is addressed in \Cref{sec:non-osc-swap}. The na\"ive idea is to make local changes near $b_1, b_2$, as in \Cref{fig:swap_local}. The main issue with this approach is that these local changes may create webs which are not fully reduced, since moves can propagate these changes through the web, see \Cref{fig:pathological_swap}. We refine this idea by carefully picking a representative element in each move-equivalence class, such that the result of the local transformation is still fully reduced.

\begin{figure*}[h!]
    \centering
    \includegraphics[width=0.4\linewidth]{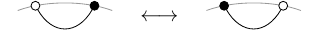}

    \centering
    \begin{subfigure}[h]{0.4\textwidth}
        \centering
        \includegraphics[width=0.9\textwidth]{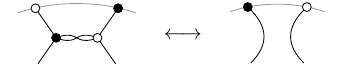}
    \end{subfigure}%
    ~ 
    \begin{subfigure}[h]{0.5\textwidth}
        \centering        \includegraphics[width=\textwidth]{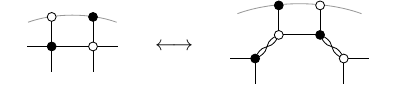}
    \end{subfigure}
    \caption{Local transformations that swap boundary conditions.}        
    \label{fig:swap_local}
\end{figure*}

\begin{figure*}[h!]
    \centering
    \includegraphics[width=0.7\linewidth]{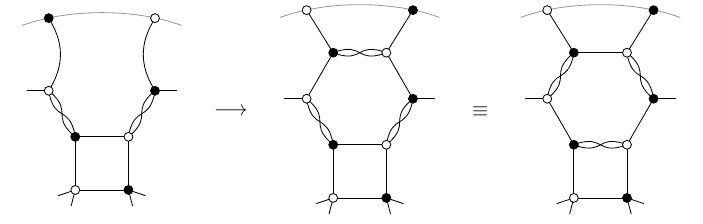}
    \caption{An example where na\"ively applying the local change from \Cref{fig:swap_local} creates a web which is not fully reduced.}        
    \label{fig:pathological_swap}
\end{figure*}

We move freely between hourglass plabic graphs of oscillating type and symmetrized six-vertex configurations for a simpler presentation of the arguments.

\subsection{Preliminary lemmas}

Throughout this subsection, let $b_1,b_2$ be adjacent boundary vertices of $D\in \WSSV(\underline a)$ with $\trip_2$-strands $\ell_1$ and $\ell_2$ such that $\ell_1$ is oriented away from $b_1$, and $\ell_2$ is oriented towards $b_2$ (this corresponds to $b_1$ being black and $b_2$ being white in the web). There are $3$ cases here:
\begin{itemize}
    \item $\ell_1=\ell_2$, or
    \item $\ell_1\ne \ell_2$ intersect, or
    \item $\ell_1$ and $\ell_2$ do not intersect.
\end{itemize}

\begin{lemma}\label{lem:singlestrand}
    If $\ell_1=\ell_2$, then $b_1$ and $b_2$ are connected in $D$ by a single edge.
\end{lemma}

\begin{proof}
    Suppose some $\trip_2$-strand $\ell$ intersects $\ell_1$. Since there are no boundary vertices between $b_1$ and $b_2$, $\ell$ must escape the region bounded by $\ell_1$ and the boundary by intersecting $\ell_1$ again. This contradicts (P1), since $D$ is well-oriented. Therefore, no $\trip_2$-strand intersects $\ell_1$, so there are no interior vertices on $\ell_1$. In other words, $b_1$ and $b_2$ are connected by a single edge.
\end{proof}

\begin{lemma}\label{lem:crossing}
    If $\ell_1\ne\ell_2$ intersect, then there exists $D'\sim D$ via Yang--Baxter moves, such that $b_1$ and $b_2$ have a common neighbor in $D'$.
\end{lemma}

\begin{proof}

    Let $q$ be the (unique) point of intersection of $\ell_1$ and $\ell_2$. Let $n_1$ be the number of vertices on $\ell_1$ strictly between $b_1$ \& $q$, and $n_2$ the number of vertices on $\ell_2$ strictly between $b_2$ \& $q$. We induct on $n = n_1 + n_2$.

    If $n = 0$, $q$ is a common neighbor of $b_1,b_2$ and we are done. So assume without loss of generality that $n_1\ge 1$. Let $q_1$ be the vertex on $\ell_1$ closest to $q$ in the direction of $b_1$. Since $b_1$ and $b_2$ are adjacent, the $\trip_2$-strand $\ell$ through $q_1$ intersects $\ell_2$ in a vertex $q_2$ between $b_2$ and $q$. We claim now that $q_2$ is adjacent to $q$. 

    Suppose for the sake of contradiction that there exists a vertex $q'$ on $\ell_2$ strictly between $q_2$ and $q$, and let $\ell'$ be the $\trip_2$-strand through $q'$. Since $D$ is fully reduced, there must be a pair of non-intersecting $\trip_2$ strands among $\ell_1,\ell_2,\ell,\ell'$. But $\ell_1,\ell_2,\ell$ form a triangle and $\ell'$ intersects $\ell_2$, so the non-intersecting pair must either be $\{\ell', \ell_1\}$ or $\{\ell', \ell\}$. If $\ell'$ does not intersect $\ell_1$, then it must end in a boundary vertex strictly between $b_1$ and $b_2$ which is impossible. On the other hand, if $\ell'$ does not intersect $\ell$, then it must intersect $\ell_1$ in a vertex strictly between $q$ and $q_1$, contradicting the choice of $q_1$. In either case, we arrive at a contradiction, and hence there are no vertices strictly between $q$ and $q_2$. Thus, $q, q_1, q_2$ form a small triangle in $D$ and using a Yang--Baxter move, we can slide $\ell$ past $q$ to obtain $D'$ and decrease $n$ by $2$ in the process. 
\end{proof}

We finally deal with the case when $\ell_1$ and $\ell_2$ do not intersect. We first introduce \emph{ladders}.

\begin{definition}
    Let $b_1,b_2,\ell_1,\ell_2$ be as above. We say that $D$ has a \emph{ladder} between $b_1$ and $b_2$ if no $\trip_2$-strand intersecting both $\ell_1, \ell_2$ has any vertices strictly between $\ell_1$ and $\ell_2$.

    If $D$ has a ladder between $b_1,b_2$ and $\ell^{(1)},\dots, \ell^{(k)}$ are the $\trip_2$-strands intersecting both $\ell_1$ and $\ell_2$ with $q_i^{(j)} = \ell_i\cap \ell^{(j)}$ for $i\in\{1,2\}, j\in [k]$, ordered so that $b_1,q_1^{(1)}, \dots, q_1^{(k)}$ appear in that order on $\ell_1$, then $b_1,q_1^{(1)}, \dots, q_1^{(k)}$ are \textit{consecutive} vertices on $\ell_1$, and $b_2,q_2^{(1)},\dots, q_2^{(k)}$ are consecutive vertices on $\ell_2$. In other words, for each $j\in[k-1]$ the vertices $q_1^{(j)}, q_1^{(j+1)}, q_2^{(j+1)}, q_2^{(j)}$ form a square, and $b_1, q_1^{(1)}, q_2^{(1)}, b_2$ form three edges of a square. We call the edges joining $q_1^{(j)}, q_2^{(j)}$ the \emph{rungs} of the ladder. 
\end{definition}

For future reference we prove the following lemma, which is a generalization of the Yang--Baxter move to slightly larger triangles:

\begin{lemma}
    \label{lem:triangle} 
    Let $\ell_1,\ell_2,\ell_3$ be $\trip_2$-strands of a well-oriented symmetrized six-vertex configuration $D$ forming a triangle $\triangle$. Assume further that the side of $\triangle$ along $\ell_1$ has length $1$, i.e., there are no vertices on $\ell_1$ strictly between $\ell_2,\ell_3$. Then the side lengths of $\triangle$ along $\ell_2$ and $\ell_3$ are the same. Moreover, after a sequence of Yang--Baxter moves sliding $\ell_3$ across, the $\trip_2$-strands passing through $\triangle$ can be made to have no vertices in the interior of $\triangle$.
\end{lemma}

\begin{proof}
    Let $\ell$ be any $\trip_2$-strand that passes through $\triangle$. Since $D$ is well oriented, $\ell$ must be parallel to one of $\ell_1,\ell_2,\ell_3$. If $\ell$ is parallel to either $\ell_2$ or $\ell_3$, then it must intersect $\ell_1$ within $\triangle$, which is impossible since the side of $\triangle$ along $\ell_1$ has length $1$. Thus, $\ell$ must be parallel to $\ell_1$. This shows that the $\trip_2$-strands passing through $\triangle$ give a bijection between the vertices on $\ell_2$ strictly between $\ell_1$ and $\ell_3$, and the vertices on $\ell_3$ strictly between $\ell_1$ and $\ell_2$. That is, the side lengths of $\triangle$ along $\ell_2$ and $\ell_3$ are the same. Now $\ell_3$ can be slid past the crossings one at a time to ensure there are no vertices in the interior of $\triangle$, see \Cref{fig:bigtriangle}.

    \begin{figure}[h]
        \centering
        \includegraphics[width=0.5\linewidth]{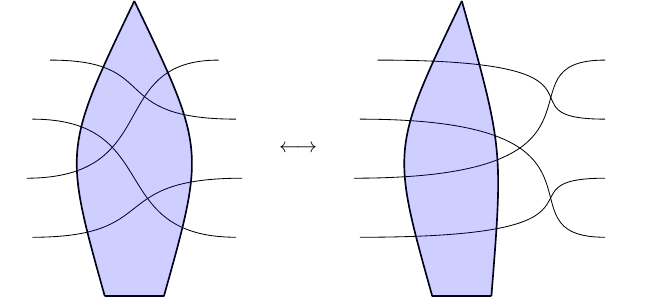}
        \caption{Sliding the right edge of the triangle past all the crossings in the interior.}
        \label{fig:bigtriangle}
    \end{figure}

\end{proof}

\begin{lemma}\label{lem:ladder}
    
    If $\ell_1$ and $\ell_2$ do not intersect, then there exists $D'\sim D$ via Yang--Baxter moves, such that $D'$ has a ladder between $b_1,b_2$.
\end{lemma}

\begin{proof}
    For the sake of convenience, draw $D$ so that the boundary vertices $b_1, b_2$ are at the top of the page and the strands $\ell_1,\ell_2$ go straight down. Order the $\trip_2$ strands $\ell^{(1)},\dots, \ell^{(k)}$ that intersect both $\ell_1,\ell_2$ based on distance of $\ell^{(i)}\cap \ell_1$ from $b_1$ so that $\ell^{(1)}\cap \ell_1$ is closer to $b_1$ than $\ell^{(k)}\cap \ell_1$. 
    
    Assume that $\ell^{(1)},\dots, \ell^{(i-1)}$ have no vertices strictly between $\ell_1,\ell_2$. We shall show that after a suitable sequence of Yang--Baxter moves, neither does $\ell^{(i)}$. We do so by inducting on the number $n$ of vertices on $\ell^{(i)}$ strictly between $\ell_1,\ell_2$. If $n=0$ there is nothing to prove, so assume $n>0$. Pick any vertex $q$ on $\ell^{(i)}$ strictly between $\ell_1,\ell_2$, and let $\ell$ be the $\trip_2$ strand through $q$. Since $\ell$ is transverse to $\ell^{(i)}$ and cannot double cross it, it must exit the region above $\ell^{(i)}$ and between $\ell_1,\ell_2$ through either $\ell_1$ or $\ell_2$. Let $\varepsilon\in\{1,2\}$ be such that $\ell$ intersects $\ell_\varepsilon$. Now $\ell$ intersects $\ell_\varepsilon$ above $\ell_\varepsilon\cap \ell^{(i)}$, so that $\ell,\ell_\varepsilon,\ell^{(i)}$ form a triangle. Let $q'$ be the vertex on $\ell_\varepsilon$ immediately above $\ell_\varepsilon\cap\ell^{(i)}$, and let $\ell'$ be the $\trip_2$ strand through $q'$. Then $\ell'$ must be parallel to either $\ell^{(i)}$ or $\ell$. If $\ell'$ is parallel to $\ell^{(i)}$, then $\ell'$ is a $\trip_2$ strand strictly above $\ell^{(i)}$ which intersects both $\ell_1,\ell_2$. But then $\ell' = \ell^{(j)}$ for some $j<i$ and $\ell'\cap \ell$ is strictly between $\ell_1,\ell_2$, contradicting our assumption.

    Thus $\ell'$ is parallel to $\ell$. Now $\ell', \ell_\varepsilon,\ell^{(i)}$ form a triangle $\triangle$ with the side along $\ell_\varepsilon$ having length $1$. Thus we may apply \Cref{lem:triangle} to first slide $\ell'$ towards $\ell^{(i)}$ and then slide all the strands passing through $\triangle$ past $q'$. Note that this does not change $n$. Finally we are left with the small triangle bounded by $\ell_\varepsilon,\ell',\ell^{(i)}$ having all sides of length $1$, applying a Yang--Baxter move we can pass the strand $\ell'$ past $\ell_\varepsilon\cap \ell^{(i)}$. This decreases $n$ by $1$. By induction, we can apply a sequence of Yang--Baxter moves to arrive at a symmetrized six-vertex configuration where $\ell^{(i)}$ has no vertices strictly between $\ell_1,\ell_2$. Continuing, we arrive at a configuration such that for each $i\in[k]$, $\ell^{(i)}$ has no vertices between $\ell_1,\ell_2$.
\end{proof}

Configurations with ladders carry different $\trip$-theoretic data based on how the rungs of the ladder are oriented. This prompts us to distinguish between two types of ladders.

\begin{definition}
    Let $D\in \WSSV(\underline a)$ with $b_1,b_2, \ell_1, \ell_2$ be as above and suppose $D$ has a ladder between $b_1, b_2$. Let $\ell^{(j)}, j\in[k]$ be the (possibly empty) collection of $\trip_2$-strands intersecting both $\ell_1$ and $\ell_2$, ordered so that $b_1,q_1^{(1)} = \ell^{(1)}\cap \ell_1,\dots, q_1^{(k)} = \ell^{(k)}\cap \ell_1$ appear in order.

    We say that the ladder is \emph{oriented} if for each $j\in[k]$, the edge is oriented $q_1^{(j)}\to q_2^{(j)}$ (by convention the ladder is oriented if $k=0$). Else the ladder is \emph{non-oriented}.
\end{definition}

\begin{figure}[h]
    \centering
    \includegraphics[height = 6cm]{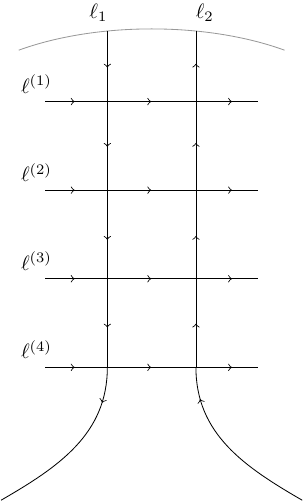}
    \caption{An oriented ladder with four rungs.}
    \label{fig:oriented_ladder}
\end{figure}

\begin{lemma}\label{lem:ladder_nonoriented}
    Let $b_1,b_2$ be adjacent vertices of opposite orientation in $D\in \WSSV(G)$ with nonintersecting $\trip_2$-strands $\ell_1,\ell_2$, and suppose $D$ has a non-oriented ladder between $b_1, b_2$. Let $q_i^{(j)}$ for $i\in\{1,2\}, j\in [k]$ be as before. Then there is a $D'\sim D$ via ASM moves, such that $b_1\to q_1^{(1)}, q_2^{(1)}\to q_1^{(1)}, q_2^{(1)} \to b_2$ are oriented edges in $D'$.
\end{lemma}

\begin{proof}
    For simplicity, assume that the orientations at the boundary are $b_1\to q_1^{(1)}$ and $q_2^{(1)}\to b_2$.
    Let $j_0$ be the smallest $j$ such that $q_2^{(j)}\to q_1^{(j)}$ is oriented. If $j_0=1$ there is nothing to prove.

    Assume $j_0>1$. We claim that we may apply an ASM move to the square with vertices $q_1^{(j_0-1)}, q_1^{(j_0)}, q_2^{(j_0)}, q_2^{(j_0-1)}$. To see this, the minimality of $j_0$ implies (by an easy induction argument) that each $1\le j<j_0$, $q_1^{(j)}\to q_2^{(j)}, q_1^{(j)}\to q_1^{(j+1)}, q_2^{(j+1)}\to q_2^{(j)}$ are oriented. Applying this to $j= j_0-1$ along with the fact that $q_2^{(j_0)}\to q_1^{(j_0)}$ is oriented ensures that the desired ASM move is applicable. Applying the ASM move gives $D'$ where $q_2^{(j_0-1)}\to q_1^{(j_0-1)}$ is oriented. Repeating this argument gives us a configuration with $q_2^{(1)}\to q_1^{(1)}$ oriented.
\end{proof}

\begin{remark}\label{rem:colorswap}
    If we swap the colors of $b_1$ and $b_2$, i.e., stipulate that $\ell_1$ is oriented towards $b_1$ and $\ell_2$ is oriented away from $b_2$, all the lemmas of this subsection still hold after reversing the orientations of all edges.
\end{remark}

\subsection{Special configurations and $\wt\swap$}

Once again, fix boundary vertices $b_1, b_2$ of $D\in \WSSV(\underline a)$ with $a_1 \ne a_2$. In this subsection, we construct a map $\wt\swap: \mathcal S(\underline a) \to \mathcal S(\underline a')$ from certain well-behaved symmetrized six-vertex configurations $D$ with boundary conditions $\underline a$, where
\[a_i' = \begin{cases}
        a_2 & \text{if } i =1,\\
        a_1 & \text{if } i = 2,\\
        a_i & \text{otherwise}.
    \end{cases}\]

\begin{definition}
    A symmetrized six-vertex configuration $D\in \WSSV(\underline a)$ is said to be \emph{special} if it satisfies one of the following properties:
    \begin{enumerate}
    \setcounter{enumi}{-1}
        \item $b_1$ and $b_2$ are connected by a simple edge, or
        \item $b_1$ and $b_2$ have a common neighbor, or
        \item $D$ has an oriented ladder between $b_1$ and $b_2$, or
        \item if $q_1, q_2$ are the neighbors of $b_1, b_2$ respectively, then the edges of the path $b_1- q_1- q_2 - b_2$ have alternating orientations in $D$.
    \end{enumerate}
    The set of special configurations with boundary conditions $\underline a$ is denoted by $\mathcal S(\underline a)\subset \WSSV(\underline a).$ Furthermore, we say that $D\in \mathcal S_i(\underline a)$ for $i\in \{0,1,2,3\}$ if it satisfies condition $(i)$ from the list above, so that $\mathcal S(\underline a) = \mathcal S_0(\underline a) \sqcup \mathcal S_1(\underline a)\sqcup \mathcal S_2(\underline a)\sqcup \mathcal S_3(\underline a)$.
\end{definition}

\begin{figure*}[h!]
    \centering
    \begin{subfigure}[h]{0.3\textwidth}
        \centering
        \includegraphics[width=0.4\textwidth]{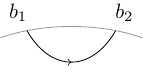}
    \end{subfigure}%
    ~
    \begin{subfigure}[h]{0.3\textwidth}
        \centering
        \includegraphics[width=0.6\textwidth]{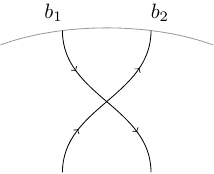}
    \end{subfigure}%
    
    \begin{subfigure}[h]{0.3\textwidth}
        \centering        
        \includegraphics[width=0.7\textwidth]{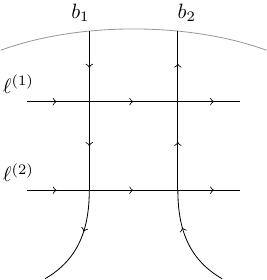}
    \end{subfigure}
    ~ 
    \begin{subfigure}[h]{0.3\textwidth}
        \centering        
        \includegraphics[width=0.7\textwidth]{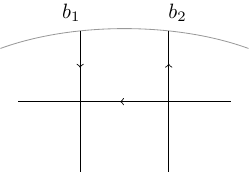}
    \end{subfigure}
    \caption{Special configurations of types (0) and (1) in the first row, and of types (2) and (3) in the second row, from left to right. Arrow reversals are also allowed.}
\end{figure*}

We are now ready to define $\wt\swap$ on the set of special configurations.

\begin{definition}
    Suppose $D\in \mathcal S(\underline a)$. Then $\wt\swap(D)$ is defined as follows:
    \begin{enumerate}
    \setcounter{enumi}{-1}
    \item If $D\in \mathcal S_0(\underline a)$, then $\wt\swap(D)$ is defined by reversing the orientation of the edge joining $b_1, b_2$.
    \item If $D\in \mathcal S_1(\underline a)$, then $\wt\swap(D)$ is defined by swapping the positions of $b_1,b_2$ in $D$ by ``undoing'' the crossing at the common neighbor $q$.

    \item If $D\in \mathcal S_2(\underline a)$, then $\wt\swap(D)$ is defined by swapping the positions of $b_1,b_2$ in $D$ by introducing a crossing between $\ell_1,\ell_2$ with the added vertex $q$ adjacent to both $b_1,b_2$.
    
    \item If $D\in \mathcal S_3(\underline a)$, then $\wt\swap(D)$ is the configuration obtained by reversing the orientations of the edges in the path $b_1-q_1- q_2- b_2$ in $D$. 
\end{enumerate}
\end{definition}

We first check that the resultant configurations are indeed well oriented. Since $\wt\swap(D)$ only introduces an extra crossing of $\trip_2$-strands $\ell_1,\ell_2$ in the case (2) wherein they do not intersect in $D$, $\wt\swap(D)$ satisfies (P1) because $D$ does. Similarly, $\wt\swap(D)$ introduces new big triangles only in case (2) wherein the ladder is oriented, so the resultant big triangles must also be oriented, showing that (P2) is also satisfied by $\wt\swap(D)$. Furthermore, it is easy to check from the definition that
\begin{enumerate}\setcounter{enumi}{-1}
    \item $D\in \mathcal S_0(\underline a)\implies \wt\swap(D)\in \mathcal S_0(\underline a')$,
    \item $D\in \mathcal S_1(\underline a) \implies \wt\swap(D)\in \mathcal S_2(\underline a')$,
    \item $D\in \mathcal S_2(\underline a)\implies \wt\swap (D)\in \mathcal S_1(\underline a')$,
    \item $D\in \mathcal S_3(\underline a)\implies \wt\swap(D)\in \mathcal S_3(\underline a')$.
\end{enumerate}
Lastly, a simple case check shows that $\wt\swap \o\wt\swap = \mathrm{id}_{\mathcal S(\underline a)}$, so $\wt\swap$ is an involution.

\begin{lemma}\label{lem:special}
    For any $D\in \WSSV(\underline a)$, there exists a $D'\sim D$ such that $D'\in \mathcal S(\underline a)$.
\end{lemma}

\begin{proof}
    Let $b_1, b_2, \ell_1, \ell_2$ in $D$ be as before. We have three cases:
    \begin{enumerate}[leftmargin=*]\setcounter{enumi}{-1}
        \item \emph{Case 0:} $\ell_1=\ell_2$. Then by \Cref{lem:singlestrand}, $b_1$ and $b_2$ are connected by a simple edge in $D$, so that $D\in \mathcal S_0(\underline a)$.
        \item \emph{Case 1:} $\ell_1\ne\ell_2$ intersect. Then by \Cref{lem:crossing} we can find a $D'\sim D$ such that $b_1, b_2$ have a common neighbor in $D'$ so that $D'\in \mathcal S_1(\underline a)$.

        \item \emph{Case 2:} $\ell_1, \ell_2$ do not intersect. Then by \Cref{lem:ladder} we can find a $D'\sim D$ such that $D'$ has a ladder between $b_1, b_2$. There are two subcases based on whether the ladder is oriented or not:
        \begin{enumerate}
            \item If the ladder in $D'$ is oriented, then $D'\in \mathcal S_2(\underline a).$
            \item If the ladder in $D'$ is non-oriented, then \Cref{lem:ladder_nonoriented} gives a $D''\sim D'$ such that $b_1 -q_1- q_2- b_2$ have alternating orientations in $D''$, ensuring that $D''\in \mathcal S_3(\underline a)$.
        \end{enumerate}
    \end{enumerate}
\end{proof}

Our next goal is to use $\wt\swap$ to define $\swap:\WSSV(\underline a)/\sim\ \to \WSSV(\underline a')/\sim$. Given any $D\in \WSSV(\underline a)$, \Cref{lem:special} guarantees that we can find a $D'\sim D$ such that $D'\in \mathcal S(\underline a)$. Thus, it is natural to try and define $\swap(D)$ to be the move-equivalence class of $\wt\swap(D')$. We check in the following subsection that this gives a well defined map: $\swap(D)$ is independent of the chosen special representative $D'\sim D$. 

\subsection{Sequences of moves}

Let $D$ be a well-oriented symmetrized six-vertex configuration and $W$ the corresponding fully reduced hourglass plabic graph. Denote by $S(D)$ the set of square faces of $D$ admitting ASM moves, and b $T(D)$ be the set of triangular faces of $D$. In the language of hourglass plabic graphs, $S(W)$ is the set of square faces of $W$ and $T(W)$ is the set of benzene faces of $D$.

If $F\in S(D)$ (resp. $F\in S(W)$), then let $s_F(D)\sim D$ (resp. $s_F(W)\sim W$) be obtained from $D$ via an ASM move at $F$ (resp. a square move at $F$). If $F\in T(D)$ (resp. $F\in T(W)$), then let $t_F(D)\sim D$ (resp. $t_F(W)\sim W$) be obtained from $D$ via a Yang--Baxter move at $F$ (resp. a benzene move at $F$). We will henceforth write $s_i$ instead of $s_{F_i}$ for $F_i\in S(D)$, and similarly $t_i$ instead of $t_{F_i}$ for $F_i\in T(D)$. 

\begin{definition}
    If $D\in \WSSV(\underline a)$, a \emph{sequence of moves on} $D$ is defined recursively as a word $w_1\cdots w_l$ such that:
    \begin{enumerate}
        \item $w_l\in \{s_F: F\in S(D)\}\sqcup \{t_F: F\in T(D)\}$, and
        \item $w_1\cdots w_{l-1}$ is a sequence of moves on $w_l(D)$.
    \end{enumerate}
    We say that $w_1\cdots w_l\sim w'_1\cdots w'_{l'}$ are \emph{equivalent} sequences of moves on $D$ if $w_1\cdots w_l(D) = w'_1\cdots w'_{l'}(D)$. For $W$ a fully reduced hourglass plabic graph, a \emph{sequence of moves on} $W$ is defined analogously.
\end{definition}

We want to study relations between different sequences of moves. We first record some simple equivalences between sequences of moves.

\begin{proposition}\label{prop:rel}
    Let $D\in \WSSV(\underline a)$ with faces $\{F_i\}_{i\in I}$. Then we have the following equivalences on sequences of moves on $D$:
    \begin{enumerate}
        \item If $F_i\in S(D), F_j\in T(s_i(D))$, then $F_j\in T(D), F_i\in S(t_j(D))$, and
        \[t_{j}s_{i} \sim s_{i}t_{j}.\]
        \item If $F_i,F_j\in S(D)$ are square faces that {do not share an edge}, then 
        \[s_i s_j\sim s_js_i.\]
        Similarly if $F_i,F_j\in T(D)$ are triangular faces that {do not share a vertex}, then 
        \[t_it_j\sim t_jt_i.\]
        \item If $F_i\in S(D), F_j\in T(D)$, then
        \[s_i^2 \sim 1, t_j^2\sim 1,\]
        where $1$ denotes the empty word.
    \end{enumerate}
\end{proposition}

\begin{proof}
    \begin{enumerate}[leftmargin=*]
        \item This relation is most easily seen in the framework of webs. Suppose $W\in \CRG(\underline a), F_i\in S(W), F_j\in T(W)$. We note that $F_i$ and $F_j$ do not share a vertex. If $F_i, F_j$ share a vertex $v$ but not an edge, then the square face $F_i$ would contribute $2$ to $\deg(v)$ and the benzene face $F_j$ would contribute $3$ to $\deg(v)$, contradicting $\deg(v) = 4$. So $F_i, F_j$ must share a whole edge. But this then gives a square face $F_i$ adjacent to the benzene face $F_j$, and after a suitable benzene move at $F_j$ we get a square with an hourglass edge, contradicting the full reducedness of $W$. Now since $F_i, F_j$ do not share a vertex, it is clear that the actions of $s_i$ and $t_j$ commute.

        \item The move $s_i$ acts by just reversing the orientations of the edges of the square $F_i$. Thus if $F_i, F_j\in S(D)$ do not share an edge, it is clear that $s_is_j(D) = s_js_i(D)$. For the statement about $t_i$ and $t_j$ we convert to the language of webs. The faces $F_i, F_j\in T(D)$ not sharing a vertex imples that $F_i, F_j\in T(W)$ do not share an edge. It is clear that benzene moves at faces that do not share an edge commute, so that $t_it_j(W) = t_jt_i(W)$.

        \item ASM moves and Yang--Baxter moves are involutions, so this relation is clear.

    \end{enumerate}
\end{proof}

\begin{corollary}[cf. Cor.~3.37 of \cite{GPPSS-sl4}]\label{cor:moves}
    If $w_1\cdots w_l$ is a sequence of moves on $D$, then $w_1\cdots w_l\sim s_{i_1}\cdots s_{i_j}t_{i_{j+1}}\cdots t_{i_l}$ for some $F_{i_1},\dots, F_{i_l}$.
\end{corollary}
\begin{proof}
    Repeatedly apply part (1) of \Cref{prop:rel}.
\end{proof}

The above corollary lets us only consider sequences of moves where we first apply all Yang--Baxter moves and then all the ASM moves (or vice versa).

\begin{remark}
    When we state the above proposition in terms of fully reduced hourglass plabic graphs, the commutation relation between square moves requires the two square faces to not share an edge, and the commutation relation between benzene moves requires the two benzene faces to not share an edge.
\end{remark}

As an easy application of the above relations, we prove that two move-equivalent special configurations must lie in the same piece $\mathcal S_i(\underline a)$.

\begin{lemma}\label{lem:special_equiv}
    Suppose $D, D'\in \mathcal S(\underline a)$. If $D\sim D'$, then $D\in \mathcal S_i(\underline a)\implies D'\in \mathcal S_i(\underline a)$ for $i\in \{0,1,2,3\}$.
\end{lemma}

\begin{proof}
    Suppose first that $D\in\mathcal S_0(\underline a)$, so that $b_1$ and $b_2$ are connected by a simple edge $e$ in $D$. Note that no ASM moves or Yang--Baxter moves affect the edge $e$. Thus $e$ is incident to $b_1$ and $b_2$ in $D'$ as well, showing that $D'\in \mathcal S_0(\underline a)$.

    Suppose next that $D\in \mathcal S_1(\underline a)$, so $\ell_1\ne\ell_2$ intersect in $D$. All moves preserve whether or not $\trip_2$-strands intersect, so $\ell_1, \ell_2$ also intersect in $D'$. Since $D'\in \mathcal S(\underline a)$, this is only possible if $D'\in \mathcal S_1(\underline a)$.

    Suppose next that $D\in \mathcal S_3(\underline a)$, so that $D$ has edges $b_1- q_1- q_2- b_2$ alternating in orientation. Then since  $D'\notin \mathcal S_0(\underline a)\sqcup \mathcal S_1(\underline a)$, so $D'$ has a ladder between $b_1, b_2$. It then suffices to show that this ladder is not oriented, so that $D'\notin \mathcal S_2(\underline a)$. Let $w$ be a sequence of moves on $D$ such that $w(D) = D'$. By \Cref{cor:moves} we may assume $w = s_{i_1}\cdots s_{i_j}t_{i_{j+1}}\cdots t_{i_l}$. Let $D'' = t_{i_{j+1}}\cdots t_{i_l}(D)$. Note that no Yang--Baxter move can involve the vertices $q_1$ or $q_2$, as any triangle containing one of these must be non-oriented. This implies that $D''$ also has edges $b_1- q_1- q_2- b_2$ with the same orientation as in $D'$. Thus $D''\in \mathcal S_3(\underline a)$. Assume for the sake of contradiction that $D'\in \mathcal S_2(\underline a)$, i.e., $D'$ has an oriented ladder between $b_1, b_2$. Since $D' = s_{i_1}\cdots s_{i_j}(D'')$ and since square moves do not change the underlying graph, $D''$ must have a ladder between $b_1, b_2$. The only ASM moves that affect the orientations of the rungs of the ladder are at the squares within the ladder. But ASM moves only apply at squares whose opposite edges have different orientations. Thus, starting at the non-oriented ladder $D''$, applying any sequence of ASM moves will still result in at least one non-oriented rung. This contradicts the fact that $D'$ is an oriented ladder. Hence $D'\in \mathcal S_3(\underline a)$.

    Lastly, since $\mathcal S(\underline a)$ is a disjoint union of the 4 pieces, $D\in \mathcal S_2(\underline a)\implies D'\in \mathcal S_2(\underline a)$.

\end{proof}

\begin{definition}
    We say that $w_1\cdots w_l$ is a \emph{non-reduced} sequence of moves on $D\in \WSSV(\underline a)$ if it is equivalent to a sequence of moves of length strictly less than $l$. Otherwise, we say that $w_1\cdots w_l$ is \emph{reduced}.
\end{definition}

We next prove two technical lemmas which say roughly that in a reduced sequence of moves, a face cannot be flipped for a second time without first flipping all its neighbors. 

\begin{lemma}\label{lem:square}
    Let $w = s_{i_1}\cdots s_{i_l}$ be a reduced sequence of ASM moves on $D$, where $F_{i_1},\dots, F_{i_l}$ are the corresponding square faces of $D$. Suppose there exists $F_0\in S(D)$ such that $s_{i_1} = s_{i_l} = s_0$, and let $F_1,\dots, F_4$ be the faces adjacent to $F_0$ in $D$. Then the moves $s_1,\dots, s_4$ appear at least once in $w$.
\end{lemma}

\begin{proof}
    We prove the lemma by contradiction. Let $w = s_{i_1}\cdots s_{i_l}$ be the smallest (by length) counterexample to the lemma with $s_{i_1} = s_{i_l} = s_0$. By minimality, $s_0$ does not appear again within $w$. Let $F_1,\dots, F_4$ be the faces of $D$ adjacent to $F_0$ labeled clockwise for convenience.

    Firstly, at least one of $s_1,\dots, s_4$ must appear in $w$. If not, all the letters in $w$ commute with $s_0$ by \Cref{prop:rel} (2) and we can cancel out the two $s_0$ letters by \Cref{prop:rel} (3), contradicting reducedness of $w$.

    Without loss of generality, assume that $s_1$ appears in $w$. If $s_1$ appears more than once in $w$, then the subword of $W$ between two consecutive appearances of $s_1$ (including the $s_1$ factors) gives a reduced sequence of moves on some $D'\sim D$ starting and ending with $s_1$ and containing no $s_0$. This gives a strictly smaller counterexample to the lemma, which is a contradiction. Thus there is precisely one instance of $s_1$ in $w$.

    Finally, let $w' = s_{i_2}\cdots s_{i_l}$. Since $w$ is a valid sequence of moves on $D$, $F_0\in S(w'(D))$. In other words an ASM move at $F_0$ must be applicable to $w'(D)$. Let $e_1$ be the edge of $F_0$ common to $F_0$ and $F_1$, and let $e_2, e_3, e_4$ be the other four edges. Since $w'$ contains precisely one instance each of $s_0$ and $s_1$, the orientation of $e_1$ in $w'(D)$ is the same as the orientation of $e_1$ in $D$. Thus an ASM move at $F_0$ is applicable to $w'(D)$ if and only if $e_2,e_3,e_4$ have the same orientations in $w'(D)$ as in $D$. Since $w'$ starts with an application of $s_0$, which reverses the orientation of $e_2,e_3,e_4$, and since $s_0$ does not appear again in $w'$, these edges can be flipped again only if $s_2, s_3, s_4$ appear in $w'$. This contradicts the fact that $w$ is a counterexample to the lemma.
\end{proof}

\begin{lemma}\label{lem:benzene}
    Let $w = t_{i_1}\cdots t_{i_l}$ be a reduced sequence of benzene moves on $W$, where $F_{i_1},\dots, F_{i_l}$ are the corresponding benzene faces of $W$. Suppose there exists $F_0\in T(W)$ such that $t_{i_1} = t_{i_l} = t_0$, and let $F_1,\dots, F_6$ be the faces adjacent to $F_0$ in $W$. Then the moves $t_1,\dots, t_6$ appear at least once in $w$.
\end{lemma}

\begin{proof}
    We prove the lemma by contradiction. Let $w = t_{i_1}\cdots t_{i_l}$ be the smallest (by length) counterexample to the lemma with $t_{i_1} = t_{i_l} = t_0$. By minimality, $t_0$ does not appear again within $w$. Let $F_1,\dots, F_6$ be the faces of $W$ adjacent to $F_0$ labeled clockwise for convenience. For each $j\in [6]$ let $e_j$ be the edge of $W$ which is the intersection of $F_0$ and $F_j$.

    At least one of $t_1,\dots, t_6$ must appear in $w$. If not, all the letters in $w$ commute with $t_0$ by \Cref{prop:rel} (2) and we can cancel out the two $t_0$ letters by \Cref{prop:rel} (3), contradicting reducedness of $w$. Without loss of generality, assume that $t_1$ appears in $w$. If $t_1$ appears more than once in $w$, then the subword of $W$ between two consecutive appearances of $t_1$ (including the $t_1$ factors) gives a reduced sequence of moves on some $W'\sim W$ starting and ending with $t_1$ and containing no $t_0$. This gives a strictly smaller counterexample to the lemma, which is a contradiction. Thus there is precisely one instance of $t_1$ in $w$.
    
    Now let $w' = t_{i_2}\cdots t_{i_l}$. Since $w$ is a valid sequence of moves on $W$, $F_0\in T(w'(W))$. In other words a benzene move at $F_0$ must be applicable to $w'(W)$. Since $w'$ contains precisely one instance each of $t_0$ and $t_1$, $e_1$ must have been a simple edge in $W$ and in $w'(W)$ (for if $e_1$ were an hourglass, then no application of $t_1$ is possible after first applying $t_0$). Thus $e_1, e_3, e_5$ are simple edges and $e_2, e_4, e_6$ are hourglass edges in $W$. For a benzene move at $F_0$ to be applicable to $w'(W)$, the edges $e_3, e_5$ must be simple edges and $e_2,e_4, e_6$ be hourglass edges in $w'(W)$. Since $w'$ starts with an application of $t_0$, it converts $e_3, e_5$ into hourglass edges and $e_2, e_4, e_6$ into simple edges. To then convert them to the desired form, since no more $t_0$ moves occur in $w'$, we have to apply $t_2, \dots, t_6$ at least once. This ensures that $w$ contains $t_1,\dots, t_6$ at least once, contradicting the fact that $w$ is a counterexample to the lemma.
\end{proof}

With these move-theoretic lemmas in place, we are now able to justify that the $\wt\swap$ map is independent of the special representative chosen:

\begin{proposition}\label{prop:special}
    Suppose $D, D'\in \mathcal S(\underline a)$ and $b_1,b_2,\ell_1,\ell_2$ are as before. If $D\sim D'$, then $\wt\swap(D) \sim \wt\swap(D')$.
\end{proposition}

\begin{proof}
    We analyze the 4 classes of special configurations separately. \Cref{lem:special_equiv} guarantees that $D, D'$ lie in the same piece $\mathcal S_i(\underline a)$ for some $i\in \{0,1,2,3\}$.

    \begin{enumerate}[leftmargin=*]\setcounter{enumi}{-1}
        \item \emph{Suppose $D,D'\in \mathcal S_0(\underline a)$}: then $b_1$ and $b_2$ are connected by a simple edge $e$ in $D$. Let $w = w_1\cdots w_l$ be a reduced sequence of moves on $D$ such that $D' = w(D)$. It is easy to see that no ASM moves or Yang--Baxter moves interact with the edge $e$. Therefore, we may apply the same sequence of moves $w$ to $\wt\swap(D)$ to obtain $\wt\swap(D')$, showing that $\wt\swap(D)\sim\wt\swap(D')$.
        
        \item \emph{Suppose $D, D'\in \mathcal S_1(\underline a)$}: Let $w = w_1\cdots w_l$ be a reduced sequence of moves on $D$ such that $D' = w(D)$. By \Cref{cor:moves} we may assume $w =  s_{i_1}\cdots s_{i_j}t_{i_{j+1}}\cdots t_{i_l}$ is a reduced sequence of moves on $D$ for some $F_1,\dots, F_l$ such that $D' = w(D)$. 
        
        Since ASM moves do not change the underlying graph, $t_{i_{j+1}}\cdots t_{i_l}(D)\in \mathcal S_1(\underline a)$. It is clear $\wt\swap$ on $\mathcal S_1(\underline a)$ commutes with ASM moves, so we may replace $D'$ by $t_{i_{j+1}}\cdots t_{i_l}$ to assume $w = t_{i_{j+1}}\cdots t_{i_l}$.

        We now use the language of webs to simplify the argument. Boundary vertices $b_1, b_2$ sharing a common neighbor in $D$ translates to the neighbors of $b_1, b_2$ in $W$ being connected by an hourglass edge $e$. Let $F_0$ be the face of $W$ incident to $e$ but not $b_1, b_2$. We claim that $t_0$ does not appear in $w$.

        Suppose for the sake of contradiction that $F_0$ is a benzene face and that $t_0$ appears in $w$. The first application of $t_0$ converts $e$ from an hourglass edge to a simple edge. But our hypothesis is that $e$ is an hourglass edge in $D'$, so there must be a second application of $t_0$. Note that the other face $F_1$ of $W$ adjacent to $e$ is clearly not a benzene face. An application of \Cref{lem:benzene} to the subword of $w$ between two $t_0$ factors yields a contradiction, proving the claim.

        The web $\wt\swap(W)$ is defined by removing the hourglass $e$ and swapping the colors of $b_1, b_2$. This operation clearly commutes with all benzene moves that happen at faces other than $F_0$. By the above claim, $w(\wt\swap(W)) = \wt\swap(w(W))$ so that $\wt\swap(W)\sim \wt\swap(W')$. Translating back to the language of six-vertex configurations, $\wt\swap(D)\sim \wt\swap(D')$.

        \item \emph{Suppose $D, D'\in \mathcal S_2(\underline a)$}: Let $w = w_1\cdots w_l$ be a reduced sequence of moves on $D$ such that $D' = w(D)$. By the description of $\wt\swap$ on $\mathcal S_2(\underline a)$ as just introducing a crossing between $\ell_1,\ell_2$ adjacent to $b_1, b_2$, it is clear that $\wt\swap$ commutes with ASM moves and with sequences of Yang--Baxter moves. Thus $w(\wt\swap(D)) = \wt\swap(w(D))$ so that $\wt\swap(D)\sim \wt\swap(D')$.

        \item \emph{Suppose $D, D'\in \mathcal S_3(\underline a)$}: Let $w = w_1\cdots w_l$ be a reduced sequence of moves on $D$ such that $D' = w(D)$. By \Cref{cor:moves} we may assume $w = s_{i_1}\cdots s_{i_j}t_{i_{j+1}}\cdots t_{i_l}$ is reduced. Since ASM moves do not change the underlying graph, $D'= w(D)$ having a ladder between $b_1,b_2$ implies that so does $t_{i_{j+1}}\cdots t_{i_l}(D)$. By the orientation of edges of the path $b_1-q_1-q_2- b_2$, it is clear that no Yang--Baxter move can slide a $\trip_2$-strand across $q_1$ or $q_2$. By the description of $\wt\swap$ on $\mathcal S_3(\underline a)$, we see that $\wt\swap(t_{i_{j+1}}\cdots t_{i_l}(D)) = t_{i_{j+1}}\cdots t_{i_l}(\wt\swap(D))$. Replacing $D$ by $t_{i_{j+1}}\cdots t_{i_l}(D)$ we may reduce to $w = s_{i_1}\cdots s_{i_l}$.

        Let $F_0$ be the face incident to the edge $q_1-q_2$ but not incident to $b_1, b_2$. We claim that $s_0$ does not appear in $w$. Assume for the sake of contradiction that $s_0$ does appear in $w$. The first application of $s_0$ flips the orientation of $e$. But by hypothesis $w(D)\in \mathcal S_3(\underline a)$, so the orientation of $e$ in $w(D)$ is the same as in $D$, implying that there exists a second factor of $s_0$ in $w$. Note that the other face $F_1$ of $D$ adjacent to $e$ is clearly not a square face. An application of \Cref{lem:square} to the subword of $w$ between two $s_0$ factors yields a contradiction, proving the claim.

        Since $\wt\swap$ is defined on $\mathcal S_3(\underline a)$ by just reversing the orientations of the edges in the path $b_1- q_1- q_2- b_2$, it is clear that $\wt\swap$ commutes with all ASM moves that happen at faces other than $F_0$. By the above claim, $w(\wt\swap(D)) = \wt\swap(w(D))$ so that $\wt\swap(D)\sim\wt\swap(D')$.
    \end{enumerate}
\end{proof}

We have finally shown the first part of the following proposition.

\begin{proposition}\label{prop:swap}
    Let $b_1, b_2$ be adjacent boundary vertices of $D\in \WSSV(\underline a)$ with $a_1\ne a_2$. Let $\underline a'$ be the boundary conditions with $a_1, a_2$ swapped. Then the map $\swap:\WSSV(\underline a)\slash\sim \ \to \WSSV(\underline a')\slash\sim$ is well defined and an involution. In particular,
    \[|\WSSV(\underline a)\slash\sim | = |\WSSV(\underline a')\slash\sim |.\]
\end{proposition}

\begin{proof}
    \Cref{prop:special} shows that $\wt\swap(D')$ for any special representative $D'\sim D$ with $D'\in \mathcal S(\underline a)$ gives the same element of $\WSSV(\underline a')/{\sim}$. Therefore $\swap$ is well defined. The fact that $\swap$ is an involution follows easily from the same statement for $\wt\swap$.
\end{proof}

\subsection{Swapping non-oscillating boundary conditions}
\label{sec:non-osc-swap}

In this subsection, we extend the results of the previous subsection to the case when the boundary conditions are not necessarily oscillating.

Let $W\in \CRG(\underline a)$ with $a_1 = 3, a_2 = 2$. We define $W'$ to be the oscillization $\osc(W)$, so $W'$ has boundary conditions $(3, 1, 1, \dots)$, with corresponding symmetrized six-vertex configuration $D'$. Note that $b_2$ and $b_3$ share a common neighbor $q$ in $D'$. We will show below how to perform local transformations at the boundary of $D'$ to construct a well-oriented symmetrized six-vertex configuration $D''$ with boundary conditions $(1,1,3,\dots)$, such that $b_1$ and $b_2$ share a common neighbor in $D''$. De-oscillizing $D''$ and converting back to webs will then give $\swap(W)$.

Let $\ell_1, \ell_2, \ell_3$ be the $\trip_2$-strands through $b_1, b_2, b_3$ respectively. We split into three cases, the first of which is the easiest to deal with:

\begin{enumerate}[leftmargin=*]\setcounter{enumi}{-1}
    \item \emph{Suppose $\ell_1=\ell_3$:} Let $\ell$ be a $\trip_2$-strand that intersects $\ell_1=\ell_3$. Then $\ell$ cannot escape the region bounded by $\ell_1$ and the boundary (since $\trip_2$-strands do not double cross), so it must terminate at the boundary between $b_1$ and $b_3$. But there is only one boundary vertex $b_2$ between $b_1$ and $b_2$, so $\ell$ must be $\ell_2$. This shows that $\ell_1=\ell_3$ consists of exactly the edges $b_3\to q\to b_1$. Thus we define $D''$ to be obtained from $D'$ by reversing the orientation of $\ell_1=\ell_3$ to be $b_1\to q\to b_3$.
\end{enumerate}
    \begin{figure}[h]
        \centering
        \includegraphics[width=0.5\linewidth]{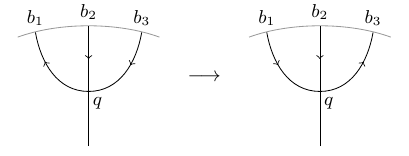}
        \caption{Case (0).}
    \end{figure}
In the other cases below, we obtain $D''$ by applying two swaps: $D'' = \swap_{2,3}\o\swap_{1,2}(D')$.
\begin{enumerate}[leftmargin=*]    
    \item \textit{Suppose $\ell_1\ne\ell_3$ do not intersect:} Suppose $\ell$ is a $\trip_2$-strand that intersects both $\ell_2$ and $\ell_3$. Then inspecting the possible configurations at $q$, we see that $\ell, \ell_2, \ell_3$ form a non-oriented triangle which contradicts the fact that $D'$ is well-oriented. Thus no $\trip_2$-strand can intersect both $\ell_2$ and $\ell_3$.

    If $\ell_1$ were to intersect $\ell_2$, then it would also intersect $\ell_3$, contradicting the previous paragraph. Thus $\ell_1, \ell_2$ do not intersect. Furthermore, any $\trip_2$-strand intersecting both $\ell_1, \ell_2$ must also intersect $\ell_3$ contradicting the claim. Hence, $D'$ has an (empty) oriented ladder between $b_1, b_2$, implying that $\swap_{1, 2}(D')$ is obtained by introducing a vertex $q'$ at which $\ell_1, \ell_2$ cross.

    Since $b_2\to q', q\to q', q\to b_3$ are oriented edges in $\swap_{1, 2}(D')$, the composition $\swap_{2, 3}(\swap_{1, 2}(D'))$ is obtained by reversing the orientations of these edges. It is clear now that $b_1, b_2$ have the common neighbor $q'$ in $D'' = \swap_{2, 3}\o\swap_{1, 2}(D)$.

    \begin{figure}[h]
        \centering
        \includegraphics[width=0.8\linewidth]{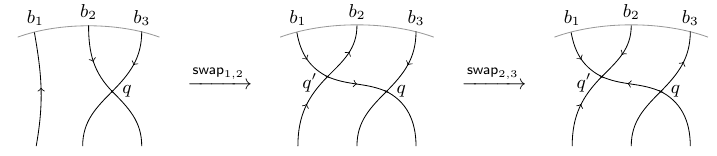}
        \caption{Case (1).}
    \end{figure}

    \item\textit{Suppose $\ell_1\ne\ell_3$ intersect:} Once again, $\ell_1$ cannot intersect $\ell_2$, as this would create a non-oriented triangle with sides $\ell_1, \ell_2, \ell_3$ by inspecting the possibilities at $q$. By \Cref{lem:ladder} we may apply a sequence of Yang--Baxter moves and assume $D'$ has a ladder between $b_1, b_2$. Since $\ell_3$ is a $\trip_2$-strand intersecting both $\ell_1, \ell_2$, by our description of ladders, $q$ is still adjacent to $b_2$ and $q' = \ell_1\cap \ell_3$ is adjacent to $b_1$. This then implies that $q$ is still adjacent to $b_3$, for if $\ell'$ were a $\trip_2$-strand intersecting $\ell_3$ between $q$ and $b_3$, then $\ell'$ has no way of escaping the region bounded by $\ell_2, \ell_3$ and the boundary of the disk. We have two subcases based on the orientation of the ladder:

    \begin{enumerate}
        \item \textit{Suppose the ladder is oriented:} Then $\swap_{1, 2}(D')$ is defined by crossing $\ell_1, \ell_2$. This creates a small triangle bounded by $\ell_1, \ell_2, \ell_3$. Applying a Yang--Baxter move at this triangle makes $q'$ a common neighbor of of $b_2, b_3$, and applying $\swap_{2, 3}$ has the effect of removing $q'$ and uncrossing $\ell_2, \ell_3$. This leaves $b_1, b_2$ with the common neighbor $q$ in $D'' = \swap_{2,3}\o\swap_{1,2}(D')$. 
        \begin{figure}[h]
        \centering
        \includegraphics[width=0.8\linewidth]{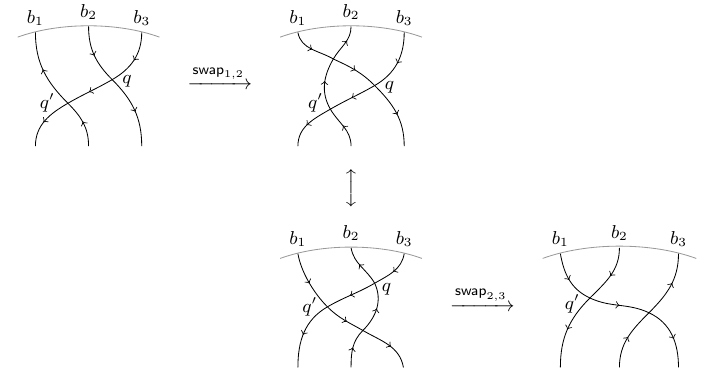}
        \caption{Case (2a).}
    \end{figure}

        \item \textit{Suppose the ladder is non-oriented:} Then we may assume (after applying suitable ASM moves by \Cref{lem:ladder_nonoriented}) that the edges $b_1\to q', q\to q', q\to b_2$ are oriented in $D$. Then $\swap_{1, 2}(D')$ is obtained by reversing the orientations of these three edges. Now $b_2, b_3$ still have the common neighbor $q$ in $\swap_{1, 2}(D')$, so applying $\swap_{2,3}$ removes $q$ and undoes the crossing between $\ell_2, \ell_3$. This leaves $b_1, b_2$ with the common neighbor $q'$ in $D'' = \swap_{2,3}\o\swap_{1,2}(D')$. \begin{figure}[h]
        \centering
        \includegraphics[width=0.8\linewidth]{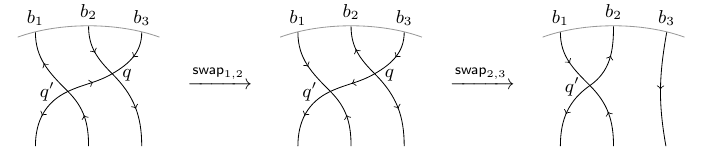}
        \caption{Case (2b).}
    \end{figure}
    \end{enumerate}
\end{enumerate}

The above allows us to extend $\swap$ to boundary conditions $a_1 = 3, a_2 = 2$. After picking suitable move-equivalent representatives, the effect of $\swap$ on webs is given by the local transformations at the boundary shown in \Cref{fig:swap_hg} below. Running the steps in reverse defines $\swap$ when $a_1 = 2, a_2 = 3$. The local transformation pictures are obtained from \Cref{fig:swap_hg} by horizontal reflection. Lastly, to define swap when $a_1 = 2, a_2 = 1$, we can just invert all colors, apply $\flip_1$, follow the same recipe as above, invert colors, and apply $\flip_2$.

\begin{figure}[h]
        \centering
        \includegraphics[width=0.5\linewidth]{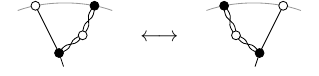}
        
        \includegraphics[width=0.5\linewidth]{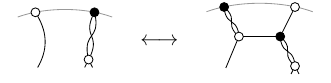}

        \includegraphics[width=0.5\linewidth]{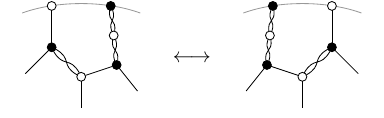}

        \includegraphics[width=0.5\linewidth]{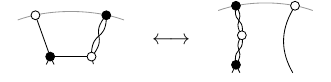}
        \caption{Effect of $\swap$ on swapping a type $2$ and type $3$ vertex. Top to bottom, these transformations correspond to cases (0), (1), (2a), (2b).}
        \label{fig:swap_hg}
        \end{figure}

\subsection{Effect of $\swap$ on trips and non-convexity}

We first analyze the effect of $\wt\swap$ on trips in the oscillating case.

\begin{lemma}
    \label{lem:swap_trip_osc}
    Let $W\in \mathcal S(\underline o)$ of oscillating type with $o_1\ne o_2$. Let $\trip_\bullet$ and $\trip'_\bullet$ be the trip permutations of $W$ and $\wt\swap(W)$ respectively, and let $(1\ 2)\in S_n$ be the transposition swapping $1$ and $2$.
    \begin{enumerate}
    \setcounter{enumi}{-1}
        \item If $W\in \mathcal S_0(\underline o)$, then
        \[\trip_i' = \trip_i \hspace{6mm}\forall i\in [3].\]
        \item If $W\in \mathcal S_1(\underline o)\sqcup \mathcal S_2(\underline o)$, then
        \begin{align*}
        \trip_1' &= \trip_1\ (1\ 2),\\
        \trip_2' &= (1\ 2)\ \trip_2\ (1\ 2),\\
        \trip_3' &= (1\ 2)\ \trip_3.            
        \end{align*}
        
        \item If $W\in \mathcal S_3(\underline o)$, then 
        \begin{align*}
        \trip_1' &= (1\ 2)\ \trip_1\ (1\ 2),\\
        \trip_2' &= \trip_2,\\
        \trip_3' &= (1\ 2)\ \trip_3\ (1\ 2).            
        \end{align*}
    \end{enumerate}
\end{lemma}

\begin{proof}
    This is an easy case check using the boundary transformations in \Cref{fig:swap_local}.
\end{proof}

Since swapping an hourglass boundary condition involves two swaps of the oscillating kind, the above lemma lets us conclude that $\swap$ preserves returning trips.

\begin{lemma}
    Let $W\in \CRG(\underline a)$ be a clasped web with $b_1,b_2$ in $\underline c_1$, and let $W'$ be a representative of $\swap(W)$ in $\CRG(\underline a')$ with $b_1,b_2$ in the clasp $\underline c_1'$. Then $W$ has a $\trip$-strand with both endpoints in $\underline c_1$ if and only if $W'$ has a $\trip$-strand with both endpoints in $\underline c_1'$.
\end{lemma}
    
\begin{proof}
    Since endpoints of trips only depend on move-equivalence classes of webs, we may safely assume that both $W$ and $W'$ are special representatives.

    If $\ell$ is a $\trip$-strand in $W$ with both endpoints in $\underline c_1$ different from $b_1$ and $b_2$, then $\ell$ also exists in $W'$ because the local transformation does not affect it. Hence, it suffices to consider when $\ell$ has one endpoint $b_1$ or $b_2$.

    In the case that both $b_1,b_2$ are incident to simple edges, \Cref{lem:swap_trip_osc} guarantees that $\ell$ transforms into a trip of $W'$ with endpoints in $\underline c_1'$ (note that the transposition $(1\ 2)$ preserves the vertices in the clasp $\underline c_1'$). In case one of $b_1, b_2$ is incident to an hourglass edge, swapping is given by partially oscillizing (after a potential flip), applying $\swap_{i, i+1}$ twice, and then de-oscillizing. The oscillizing and de-oscillizing steps ensure that the endpoints of the trip remain within the clasp, and so do the two $\swap_{i, i+1}$ applications by \Cref{lem:swap_trip_osc}.
\end{proof}

Next we analyze the effect of $\wt\swap$ on non-convexity in the oscillating case.

\begin{lemma}
    \label{lem:swap_nc_osc}
    Let $W\in \mathcal S(\underline o)$ with $o_1\ne o_2$, and $\underline C$ be a clasp sequence such that $b_1,b_2$ are in the same clasp $\underline c_1$. Then for any minimal cut path $\gamma$ of $\underline c_1$ in $W$, there is a cut path $\gamma'$ of $W'=\wt\swap(W)$ with the same weight. In particular, $W$ is non-convex if and only if $W'$ is.
\end{lemma}

\begin{proof}
    The case when $W\in S_0(\underline o)$ is trivial. Next we analyze the case when $W\in S_1(\underline o)\sqcup S_2(\underline o)$. Then $\wt\swap$ is given by the local transformation in \Cref{fig:swap_local_12}.

    \begin{figure}[h]
        \centering
        \includegraphics[height=1.5cm]{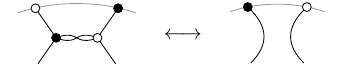}
        \caption{Local transformation for the effect of $\swap$ on $W\in S_1(\underline o)\sqcup S_2(\underline o)$.}
        \label{fig:swap_local_12}
    \end{figure}

    Note that a minimal cut path for $\underline c_1$ cannot intersect the hourglass edge shown in the figure on the left, and hence minimal cut paths of the left figure correspond easily to minimal cut paths of the right figure.

    In the case when $W\in S_3(\underline o)$, $\wt\swap$ is given by the local modification (after uncontracting the 4-valent vertices) in \Cref{fig:swap_local_3}.

    \begin{figure}[h]
        \centering
        \includegraphics[height=3cm]{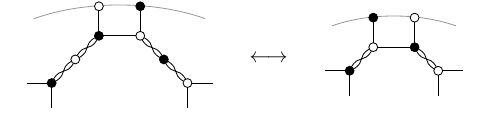}
        \caption{Local transformation for the effect of $\swap$ on $W\in S_3(\underline o)$.}
        \label{fig:swap_local_3}
    \end{figure}
    It is easy to see here that minimal cut paths in the left figure correspond to minimal cut paths in the right figure.
\end{proof}

Using the above, we can show that general swaps (involving type $2$ boundary conditions) also preserve non-convexity

\begin{lemma}
    \label{lem:swap_nc}
    Let $W\in \CRG(\underline a)$ be a clasped web with $b_1, b_2$ in $\underline c_1$, and let $W'$ be a representative of $\swap(W)$ in $\CRG(\underline a')$ with $b_1,b_2$ in the induced clasp $c_1'$. Then $W$ is non-convex if and only if $W'$ is non-convex.
\end{lemma}

\begin{proof}
    If neither $b_1$ nor $b_2$ is incident to an hourglass, then we are done by \Cref{lem:swap_nc_osc}. Suppose without loss of generality that $b_1$ is incident to an hourglass. Then after picking suitable move-equivalent representatives, $W'$ is obtained from $W$ by first partially oscillizing to get $\osc_1(W)$, performing $\swap_{2,3}$ then $\swap_{1,2}$, and finally de-oscillizing.

    Note that the oscillizing and de-oscillizing steps do not affect weights of minimal cut paths. More precisely, every minimal cut path of $W$ corresponds to a minimal cut path of $\osc_1(W)$ of the same weight. Applying \Cref{lem:swap_nc_osc} twice then shows that the minimal cut weights of $W$ are the same as the minimal cut weights of $W'$, and hence $W$ is non-convex if and only if $W'$ is. 
\end{proof}

\begin{remark}
    We can also prove that $\swap$ preserves returning trips and non-convexity for hourglass boundary conditions using the explicit description of $\swap$ via local boundary transformations as in \Cref{fig:swap_hg}.

\end{remark}

In summary, we have proven that the map $\swap$ gives a bijection that preserves returning non-convexity and returning trips. Iterated application of $\swap$ lets us sort the boundary conditions of any given web.

\begin{corollary}\label{cor:sorting}
    Let $\underline C = (\underline c_1, \dots, \underline c_m), \underline C' = (\underline c_1', \dots, \underline c_m')$ be clasp sequences on $\underline a, \underline a'$ respectively, such that each $\underline c_i'$ is a reordering of the tuple $\underline c_i$. There is a bijection
    \[\Psi: \CRG(\underline a)\to \CRG(\underline a')\]
    such that 
    \begin{itemize}
        \item $W\in \CRG(\underline a)$ is non-convex with respect to $\underline C$ if and only if $\Psi(W)$ is non-convex with respect to $\underline C'$.
        \item $W\in \CRG(\underline a)$ has no trips starting and ending in the samef clasp of $\underline C$ if and only if $\Psi(W)$ has the same property with respect to $\underline C'$.
    \end{itemize}
\end{corollary}

We are finally able to prove our main theorem for general clasp sequences.

\begin{thm}
    Let $\underline C = (\underline c_1, \dots, \underline c_m)$ be a clasp sequence on $\underline a\in [3]^n$, with $\weight(\underline c_i) = \lambda_i$. Then a basis for $\inv_G\left(\bigotimes_{i=1}^m V(\lambda_i)\right)$ is given by
    \[\mathcal W_{\underline C}\coloneqq \{\pi_{\underline C}([W]): [W]\in \mathcal W_{\underline a}\setminus \ker \pi_{\underline C}\}.\]
    Moreover, the following are equivalent:
    \begin{enumerate}
        \item $[W]\notin \ker\pi_{\underline C}$,
        \item $W$ is non-convex,
        \item $W$ has no trips that start and end in the same clasp.
    \end{enumerate}
    \label{thm:main_general}
\end{thm}

\begin{proof}
    It is clear that $\mathcal W_{\underline C}$ spans $\inv_G\left(\bigotimes_{i=1}^m V(\lambda_i)\right)$, so it suffices to show that $|\mathcal W_{\underline C}|\le \dim\inv_G\left(\bigotimes_{i=1}^m V(\lambda_i)\right)$. By \Cref{prop:nonconvex}, (1)$\implies$(2), so
    \[|\mathcal W_{\underline C}| \le |\{W\in \mathcal W_{\underline a}: W\text{ is non-convex}\}|.\]
    
    Let $\underline C' = (\underline c_1', \dots, \underline c_m')$ be the sorted clasp sequence on $\underline a'$ obtained by sorting each of the $\underline c_i$, and let $\Psi: \CRG(\underline a)\to \CRG(\underline a')$ be as in \Cref{cor:sorting}. Since $\Psi$ maps non-convex webs bijectively onto non-convex webs,
    \[|\{W\in \mathcal W_{\underline a}:W\text{ is non-convex}\}| = |\{W'\in \mathcal W_{\underline a'}:W'\text{ is non-convex}\}|.\]
    
    By \Cref{thm:sorted} for sorted clasps, the set of non-convex webs in $\mathcal W_{\underline a'}$ is precisely $\mathcal W_{\underline C'}$, and
    \[|\mathcal W_{\underline C'}| = \dim\inv_G\left(\bigotimes_{i=1}^m V(\lambda_i)\right).\]
    Therefore $|\mathcal W_{\underline C}|\le \dim\inv_G\left(\bigotimes_{i=1}^m V(\lambda_i)\right)$, showing that $\mathcal W_{\underline C}$ forms a basis for $\inv_G\left(\bigotimes_{i=1}^m V(\lambda_i)\right)$. This also implies that $\mathcal W_{\underline C} = \{W\in \mathcal W_{\underline a}: W\text{ is non-convex}\}$, so that (2)$\implies$(1).

    It remains to check that (3) is equivalent to (2). By \Cref{cor:sorting}, $W$ is non-convex if and only if $\Psi(W)$ is non-convex, and $W$ has no trips returning to the same clasp if and only if $\Psi(W)$ has the same property. Moreover, \Cref{thm:sorted} implies that $\Psi(W)$ is non-convex if and only if $\Psi(W)$ has no trips returning to the same clasp. This shows the equivalence of items (2) and (3) in the theorem.
\end{proof}

Returning to the quantum case, the character of the finite-dimensional simple $\uq$ module $V_q(\lambda_i)$ (of \emph{type} $\mathbf{1}$, in the sense of e.g. \cite[Ch.~5]{jantzen-book}) is the same as that of the classical module $V(\lambda_i)$ that it deforms (working over $\C(q)$, so that $q$ is generic). Therefore, \Cref{prop:nonconvex} is valid even in the quantum case, by the same proof. This shows that $\W_{\underline C}$ spans $\inv_{U_q(\mathfrak{sl}_4)}\big(\bigotimes_{i=1}^m V_q(\lambda_i)\big)$, where $\lambda_i = \weight(\underline c_i)$. Moreover, the decomposition of tensor product $\bigotimes_{i=1}^m V_q(\lambda_i)$ into irreducibles is the same as in the classical case. Therefore, we still have that
\[\dim \inv_{U_q(\mathfrak{sl}_4)}\left(\bigotimes_{i=1}^m V_q(\lambda_i)\right) = |\RT(\underline C)|.\] 
The proof of \Cref{thm:sorted} now goes through without any modification. Finally, the sorting bijection and \Cref{cor:sorting} prove \Cref{thm:main_general} in the quantum case.

\section{The $r=2,3$ cases}
\label{sec:r-2-3}

In this section, we recall Kuperberg's \cite{Kuperberg} clasped $\SL_2$ and $\SL_3$ web bases and prove their equivalence with our characterizations in terms of trips, boundary configurations, and descents. Some proofs are merely sketched, as they follow the same plan as the (much harder) proof for $r=4$ above.

\subsection{The $r=2$ case}

Kuperberg defined in \cite{Kuperberg} $\SL_2$ webs to be non-crossing matchings drawn within a disk. We fit this into the plabic framework by coloring every boundary vertex black, and introducing an internal white vertex along each strand of the matching. This gives a plabic graph such that all
\begin{itemize}
    \item boundary vertices are black, and
    \item interior vertices are white and of degree $2$.
\end{itemize}

Kuperberg showed that clasped webs given by non-crossing matchings with no U-turns, as in \Cref{fig:u_turn}, form a basis for the corresponding invariant space.

Similar to the $r=4$ case, the separation labeling of a web $W$ gives a lattice word $\mathcal L(W)$ with letters in $\{1, 2\}$. A descent of a lattice word is defined to be a subword of the form $12$.

If we declare U-turns to be the only bad local configurations, we obtain \Cref{thm:intro-sorted} for $r=2$. Note that since all boundary vertices are declared to be black, all clasps are already sorted.

\begin{figure}[h]
    \centering
    \includegraphics[width=0.2\linewidth]{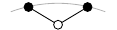}
    \caption{U-turn in an $\SL_2$ web.}
    \label{fig:u_turn}
\end{figure}

\subsection{The $r=3$ case}

Similar to the $r=2$ case, much of the work was already done by Kuperberg \cite{Kuperberg}, who, however, did not consider trips.

Thinking of $\SL_3$ webs as plabic graphs, we can consider their trips. Note that here we have only $\trip_1$ and $\trip_2$, which are inverses. Using trips, we can construct the separation labeling to obtain a lattice word $\mathcal L(W) \coloneqq \partial\sep_W$ with letters in $[3]\sqcup\overline{[3]}$. A \emph{sorted} word would contain all unbarred letters before all barred letters. We define descents to be subwords of the form
\begin{itemize}
    \item $ab$ where $a<b$ for $a,b\in[3]$, or
    \item $\overline {ab}$ where $a>b$ for $a,b\in[3]$, or
    \item $1\overline 1$.
\end{itemize}
This allows us to define for every clasp sequence $\underline C$ the set $\BL(\underline C)$ of balanced lattice words with no $\underline C$-descents.

The inverse of the map $W\mapsto\L(W)$ is given by the Khovanov--Kuperberg growth rules from \cite{Khovanov-Kuperberg}. The growth algorithm here works similarly to \Cref{algo:growth} where we place all boundary vertices on a horizontal line with downward dangling strands labelled by letters of the word, and apply the growth rules from \Cref{fig:kk_growth} till no dangling strands remain.

\begin{figure}[h]
    \centering
    \includegraphics[width=0.7\linewidth]{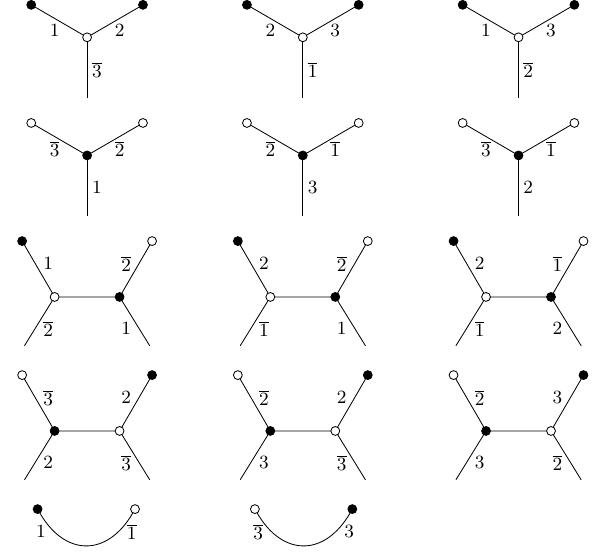}
    \caption{Khovanov--Kuperberg growth rules.}
    \label{fig:kk_growth}
\end{figure}

\begin{remark}
    In \cite{Khovanov-Kuperberg}, lattice words consist of the \emph{sign and state} letters from the set $\{+,-\}\times \{-1, 0, 1\}$. The translation to lattice words in our setting is given by 
    \begin{align*}
        (+, 1) \longleftrightarrow 1, \quad\quad & (-, 1) \longleftrightarrow \overline 3,\\
        (+, 0) \longleftrightarrow 2, \quad\quad & (-, 0) \longleftrightarrow \overline 2,\\
        (+, -1) \longleftrightarrow 3, \quad\quad & (-, -1) \longleftrightarrow \overline 1.\\
    \end{align*}
\end{remark}

It is shown in \cite[Thm~3.2]{LACIM} that the inverse to the KK-growth algorithm produces the proper edge coloring that is lexicographically minimal. Since the separation labeling is the lexicographically minimal labeling, the inverse to the KK growth algorithm is the map $W\mapsto \partial\sep(W) = \L(W)$, as claimed. We will use the KK-growth rules to show that descents in $\L(W)$ correspond to one of the bad local boundary configurations in \Cref{fig:badconf_3}.

\begin{figure}[h]
    \centering
    \includegraphics[width=0.7\linewidth]{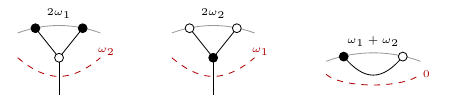}
    \caption{Bad local boundary configurations for $\SL_3$ webs, with partial convexity witnessed by the dashed arc.}
    \label{fig:badconf_3}
\end{figure}

With this setup, the main theorem for sorted clasps is true here as well, and we sketch how to adapt the proof of \Cref{thm:sorted}.

\begin{thm}\label{thm:sorted3}
    Fix boundary conditions $\underline a \in [2]^n$ and a {sorted} clasp sequence $\underline C = (\underline c_1,\dots, \underline c_m)$, and let $\lambda_i = \weight(\underline c_i)$. Then a basis for $\inv_G\left(\bigotimes_{i=1}^m V(\lambda_i)\right)$ is given by
    \[\mathcal W_{\underline C} \coloneqq \{\pi_{\underline C}([W]): [W]\in \mathcal W_{\underline a}\setminus \ker\pi_{\underline C}\}.\]
    Moreover, the following are equivalent for $W\in \mathcal W_{\underline a}$:
    \begin{enumerate}
        \item $[W] \notin \ker\pi_{\underline C}$.
        \item $W$ is non-convex.
        \item There are no local boundary configurations from \Cref{fig:badconf_3} occurring within any clasp of $W$.
        \item The lattice word $\L(W) = \partial\sep_W$ has no $\underline C$-descents.
        \item $W$ has no trips that start and end in the same clasp.
    \end{enumerate}

    \begin{proof}
    By \Cref{prop:nonconvex}, $\W_{\underline C}$ spans the invariant space, and we show that $|\W_{\underline C}|\le \dim\inv_G\left(\bigotimes_{i=1}^m V(\lambda_i)\right)$ while simultaneously proving the equivalence of (1), (2), (3) and (4).(1)$\implies$ (2) is \Cref{prop:nonconvex}.(2)$\implies$(3) is seen directly from \Cref{fig:badconf_3}. (3)$\implies$(4) is proven exactly like \Cref{prop:descent}, using the KK-growth rules instead. There are far fewer cases here, since each descent is of one of the following types:
    \begin{itemize}
        \item $ab$ with $a<b$,
        \item $\overline{ab}$ with $a>b$, or
        \item $1\overline 1$.
    \end{itemize}

    Now (1)$\implies$(4) implies that $\W_{\underline C}$ injects into $\BL(\underline C)$. Just as in \Cref{prop:LR_bijection}, we can construct a bijection between balanced lattice words with no $\underline C$-descents and $\weight(\underline C)$-Littlewood Richardson tableaux, showing that $|\BL(\underline C)| = \dim\inv_G\left(\bigotimes_{i=1}^m V(\lambda_i)\right)$. This proves the first statement, along with the equivalence of (1), (2), (3) and (4).

    Bad local boundary configurations all have returning trips, so (5)$\implies$(3). Lastly if $W$ is non-convex, we show using (4) and the definition of $\sep_W$ to show that there are no returning trips within the \textit{first} clasp, mimicking the proof of \Cref{thm:sorted}. By rotation invariance, this proves the equivalence of (5) with non-convexity.
        
    \end{proof}
\end{thm}

To generalize to non-sorted clasps, we can use Kuperberg's H-webs \cite{Kuperberg} to construct our $\swap$ map as follows:
\begin{itemize}
    \item If $b_1,b_2$ are connected via a length $3$ path, then an H-web is deleted.
    \item Else, an H-web is added.
\end{itemize}

\begin{figure}[h]
    \centering
    \includegraphics[height = 1.5cm]{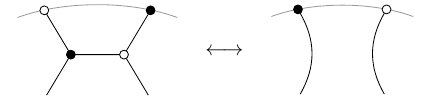}
    \label{fig:h}
    \caption{$\swap$ for $r=3$. Color reversals are also allowed.}
\end{figure}
It is easy to see, as in the $r=4$ case, that $\swap$ preserves non-convexity and returning trips. This gives us the theorem for general, non-sorted clasps.

\begin{thm}
    Let $\underline C = (\underline c_1, \dots, \underline c_m)$ be a clasp sequence on $\underline a\in [2]^n$, with $\weight(\underline c_i) = \lambda_i$. Then a basis for $\inv_G\left(\bigotimes_{i=1}^m V(\lambda_i)\right)$ is given by
    \[\mathcal W_{\underline C}\coloneqq \{\pi_{\underline C}([W]): [W]\in \mathcal W_{\underline a}\setminus \ker \pi_{\underline C}\}.\]
    Moreover, the following are equivalent:
    \begin{enumerate}
        \item $[W]\notin \ker\pi_{\underline C}$,
        \item $W$ is non-convex,
        \item $W$ has no trips that start and end in the same clasp.
    \end{enumerate}
\end{thm}

\section*{Acknowledgements}

We thank Ben Elias for his thought-provoking comments and Elise Catania, Jesse Kim, and Stephan Pfannerer for letting us know about their independent work in progress \cite{catania-kim-pfannerer}. The writing of this paper was completed during our visit to ICERM for the workshop ``Webs in Algebra, Geometry, Topology, and Combinatorics", and we are very grateful for the excellent collaborative conditions provided there.

\bibliographystyle{amsalphavar}
\bibliography{main}

\newcommand{\etalchar}[1]{$^{#1}$}
\providecommand{\bysame}{\leavevmode\hbox to3em{\hrulefill}\thinspace}
\providecommand{\MR}{\relax\ifhmode\unskip\space\fi MR }
\providecommand{\MRhref}[2]{%
  \href{http://www.ams.org/mathscinet-getitem?mr=#1}{#2}
}
\providecommand{\href}[2]{#2}
\begin{thebibliography}{GPPSS25b}

\bibitem[BDG{\etalchar{+}}22]{LACIM}
V\'{e}ronique Bazier{-}Matte, Guillaume Douville, Alexander Garver, Rebecca
  Patrias, Hugh Thomas, and Emine Y{\i}ld{\i}r{\i}m, \emph{Leading terms of
  {$SL_3$} web invariants}, Int. Math. Res. Not. IMRN (2022), no.~3,
  1714--1733.

\bibitem[CKM14]{Cautis-Kamnitzer-Morrison}
Sabin Cautis, Joel Kamnitzer, and Scott Morrison, \emph{Webs and quantum skew
  {H}owe duality}, Math. Ann. \textbf{360} (2014), no.~1-2, 351--390.

\bibitem[CKP]{catania-kim-pfannerer}
Elise Catania, Jesse Kim, and Stephan Pfannerer, \emph{Clasped webs and
  promotion of non-rectangular tableaux}, in preparation.

\bibitem[DKS24]{Douglas-Kenyon-Shi}
Daniel~C. Douglas, Richard Kenyon, and Haolin Shi, \emph{Dimers, webs, and
  local systems}, Trans. Amer. Math. Soc. \textbf{377} (2024), no.~2, 921--950.

\bibitem[FLL19]{Fraser-Lam-Le}
Chris Fraser, Thomas Lam, and Ian Le, \emph{From dimers to webs}, Trans. Amer.
  Math. Soc. \textbf{371} (2019), no.~9, 6087--6124.

\bibitem[FP16]{Fomin-Pylyavskyy-advances}
Sergey Fomin and Pavlo Pylyavskyy, \emph{Tensor diagrams and cluster algebras},
  Adv. Math. \textbf{300} (2016), 717--787.

\bibitem[GPPSS24]{fluctuating-paper}
Christian Gaetz, Oliver Pechenik, Stephan Pfannerer, Jessica Striker, and
  Joshua~P. Swanson, \emph{Promotion permutations for tableaux}, Combin. Theory
  \textbf{4} (2024), 56 pages.

\bibitem[GPPSS25a]{GPPSS-sl4}
Christian Gaetz, Oliver Pechenik, Stephan Pfannerer, Jessica Striker, and
  Joshua~P. Swanson, \emph{Rotation-invariant web bases from hourglass plabic
  graphs}, Invent. Math. (2025).

\bibitem[GPPSS25b]{two-column}
Christian Gaetz, Oliver Pechenik, Stephan Pfannerer, Jessica Striker, and
  Joshua~P. Swanson, \emph{Web bases in degree two from hourglass plabic
  graphs}, Int. Math. Res. Not. IMRN (2025), no.~13, rnaf189, 23 pages.

\bibitem[Hag18]{Hagemeyer}
Colin~Scott Hagemeyer, \emph{Spiders and generalized confluence}, ProQuest LLC,
  Ann Arbor, MI, 2018, Thesis (Ph.D.)--University of California, Davis.

\bibitem[Jan96]{jantzen-book}
Jens~Carsten Jantzen, \emph{Lectures on quantum groups}, Graduate Studies in
  Mathematics, vol.~6, American Mathematical Society, Providence, RI, 1996.

\bibitem[Kho04]{Khovanov}
M.~Khovanov, \emph{{sl}(3) link homology}, Algebr. Geom. Topol. \textbf{4}
  (2004), 1045--1081.

\bibitem[KK99]{Khovanov-Kuperberg}
Mikhail Khovanov and Greg Kuperberg, \emph{Web bases for {${\rm sl}(3)$} are
  not dual canonical}, Pacific J. Math. \textbf{188} (1999), no.~1, 129--153.

\bibitem[Kup96]{Kuperberg}
Greg Kuperberg, \emph{Spiders for rank {$2$} {L}ie algebras}, Comm. Math. Phys.
  \textbf{180} (1996), no.~1, 109--151.

\bibitem[LS24]{Le.Sikora}
Thang T.~Q. L\^e and Adam~S. Sikora, \emph{Stated {${\rm SL}(n)$}-skein modules
  and algebras}, J. Topol. \textbf{17} (2024), no.~3, Paper No. e12350, 93.

\bibitem[Lus90]{Lusztig:canonical}
G.~Lusztig, \emph{Canonical bases arising from quantized enveloping algebras},
  J. Amer. Math. Soc. \textbf{3} (1990), no.~2, 447--498.

\bibitem[Pos18]{Postnikov-ICM}
Alexander Postnikov, \emph{Positive {G}rassmannian and polyhedral
  subdivisions}, Proceedings of the {I}nternational {C}ongress of
  {M}athematicians---{R}io de {J}aneiro 2018. {V}ol. {IV}. {I}nvited lectures,
  World Sci. Publ., Hackensack, NJ, 2018, pp.~3181--3211.

\bibitem[PPR09]{Petersen-Pylyavskyy-Rhoades}
T.~Kyle Petersen, Pavlo Pylyavskyy, and Brendon Rhoades, \emph{Promotion and
  cyclic sieving via webs}, J. Algebraic Combin. \textbf{30} (2009), no.~1,
  19--41.

\bibitem[Sik05]{Sikora}
Adam~S. Sikora, \emph{Skein theory for {${\rm SU}(n)$}-quantum invariants},
  Algebr. Geom. Topol. \textbf{5} (2005), 865--897.

\end{thebibliography}
\end{document}